\documentclass{amsart}
\usepackage{fullpage}
\usepackage[mathscr]{euscript}
\usepackage{mathpazo}

\usepackage{amsmath,amssymb}
\usepackage{hyperref}
\usepackage{mathtools}
\usepackage{tikz-cd}
\usepackage{shuffle}
\usepackage{multicol}
\usepackage{enumitem}
\usepackage{bbm}
\usepackage{booktabs}

\usepackage{amsthm}
\usepackage{graphicx}
\usepackage{cleveref}

\usepackage[multiple]{footmisc}

\usepackage{lipsum}
\usepackage{amsfonts}
\usepackage{graphicx}
\usepackage{epstopdf}
\usepackage{algorithmic}
\ifpdf
  \DeclareGraphicsExtensions{.eps,.pdf,.png,.jpg}
\else
  \DeclareGraphicsExtensions{.eps}
\fi

\usepackage{enumitem}
\setlist[enumerate]{leftmargin=.5in}
\setlist[itemize]{leftmargin=.5in}

\theoremstyle{definition}
\newtheorem{definition}{Definition}[section]
\newtheorem{example}{Example}[section]

\theoremstyle{plain}
\newtheorem{theorem}{Theorem}[section]
\newtheorem{proposition}{Proposition}[section]
\newtheorem{lemma}{Lemma}[section]
\newtheorem{corollary}{Corollary}[section]

\theoremstyle{remark}
\newtheorem{remark}{Remark}[section]

\title{Signatures, Lipschitz-free spaces, and paths of persistence diagrams}

\author{Chad Giusti}
\address[Chad Giusti]{Department of Mathematics, Oregon State University, Corvallis, OR 97331}
\email{chad.giusti@oregonstate.edu}
\author{Darrick Lee}
\address[Darrick Lee]{Mathematical Institute, University of Oxford, Andrew Wiles Building, Radcliffe Observatory Quarter, Woodstock Rd, Oxford OX2 6GG}
\email{darrick.lee@maths.ox.ac.uk}

\usepackage{amsopn}

\numberwithin{equation}{section}

\newcommand{\ps}[1]{\mkern-.25mu\mathbin{(\mkern-3.5mu({#1})\mkern-3.5mu)}}

\newcommand{\Z}{\mathbb{Z}}
\newcommand{\N}{\mathbb{N}}

\newcommand{\R}{\mathbb{R}}

\newcommand{\PP}{\mathbb{P}}

\newcommand{\E}{\mathbb{E}}
\newcommand{\X}{\mathbb{X}}

\newcommand{\bx}{\mathbf{x}}

\newcommand{\br}{\mathbf{r}}
\newcommand{\bs}{\mathbf{s}}

\newcommand{\bv}{\mathbf{v}}

\newcommand{\bX}{\mathbf{X}}

\newcommand{\sPM}{\mathsf{PM}}

\newcommand{\cR}{\mathcal{R}}
\newcommand{\cP}{\mathcal{P}}

\newcommand{\cF}{\mathcal{F}}

\newcommand{\cD}{\mathcal{D}}
\newcommand{\cH}{\mathcal{H}}
\newcommand{\cM}{\mathcal{M}}

\newcommand{\cX}{\mathcal{X}}

\newcommand{\cV}{\mathcal{V}}
\newcommand{\cG}{\mathcal{G}}

\newcommand{\fin}{\textrm{fin}}

\newcommand{\Nagent}{N_{\mathrm{agent}}}
\newcommand{\Ntrial}{N_{\mathrm{trial}}}
\newcommand{\Nsimulation}{N_{\mathrm{simulation}}}

\newcommand{\Pers}{\mathrm{Pers}}
\newcommand{\Cpl}{\mathrm{Cpl}}

\newcommand{\Lip}{\mathrm{Lip}}

\newcommand{\Seq}{\mathrm{Seq}}
\newcommand{\VR}{\mathrm{VR}}
\newcommand{\PH}{\mathrm{PH}}
\newcommand{\im}{\mathrm{im}}
\newcommand{\PD}{\mathrm{PD}}

\newcommand{\wOmega}{\widetilde{\Omega}}
\newcommand{\wid}{\widetilde{d}}
\newcommand{\wP}{\widetilde{P}}
\newcommand{\bB}{\mathbf{B}}

\newcommand{\generalspace}{Z}
\newcommand{\qgeneralspace}{\widetilde{\generalspace}}

\begin{document}
\begin{abstract}
    Paths of persistence diagrams provide a summary of the dynamic topological structure of a one-parameter family of metric spaces. These summaries can be used to study and characterize the dynamic shape of data such as swarming behavior in multi-agent systems, time-varying fMRI scans from neuroscience, and time-dependent scalar fields in hydrodynamics. While persistence diagrams can provide a powerful topological summary of data, the standard space of persistence diagrams lacks the sufficient algebraic and analytic structure required for many theoretical and computational analyses. We enrich the space of persistence diagrams by isometrically embedding it into a \emph{Lipschitz-free space}, a Banach space built from a universal construction. We utilize the Banach space structure to define bounded variation paths of persistence diagrams, which can be studied using the \emph{path signature}, a reparametrization-invariant characterization of paths valued in a Banach space. The signature is \emph{universal} and \emph{characteristic}, which allows us to theoretically characterize measures on the space of paths and motivates its use in the context of kernel methods. However, kernel methods often require a feature map into a Hilbert space, so we introduce the \emph{moment map}, a stable and injective feature map for static persistence diagrams, and compose it with the discrete path signature, producing a computable feature map into a Hilbert space. Finally, we demonstrate the efficacy of our methods by applying this to a parameter estimation problem for a 3D model of swarming behavior. 
\end{abstract}

\maketitle

\section{Introduction}

In complex systems, collective behavior emerges from the interactions between individual elements~\cite{dorsogna_self-propelled_2006, nguyen_thermal_2012}, but the state of an individual element often contains very little information. It is therefore reasonable, and even desirable, to describe the state of the system by using abstract or qualitative organizational features. The negligible contribution of individual agents suggests that it should be possible to infer such features from a relatively small representative subsample of the constituent agents, and the development of tools for effectively doing so is an area of active research. 

\begin{figure}
    \centering
            \includegraphics[width=0.98\textwidth]{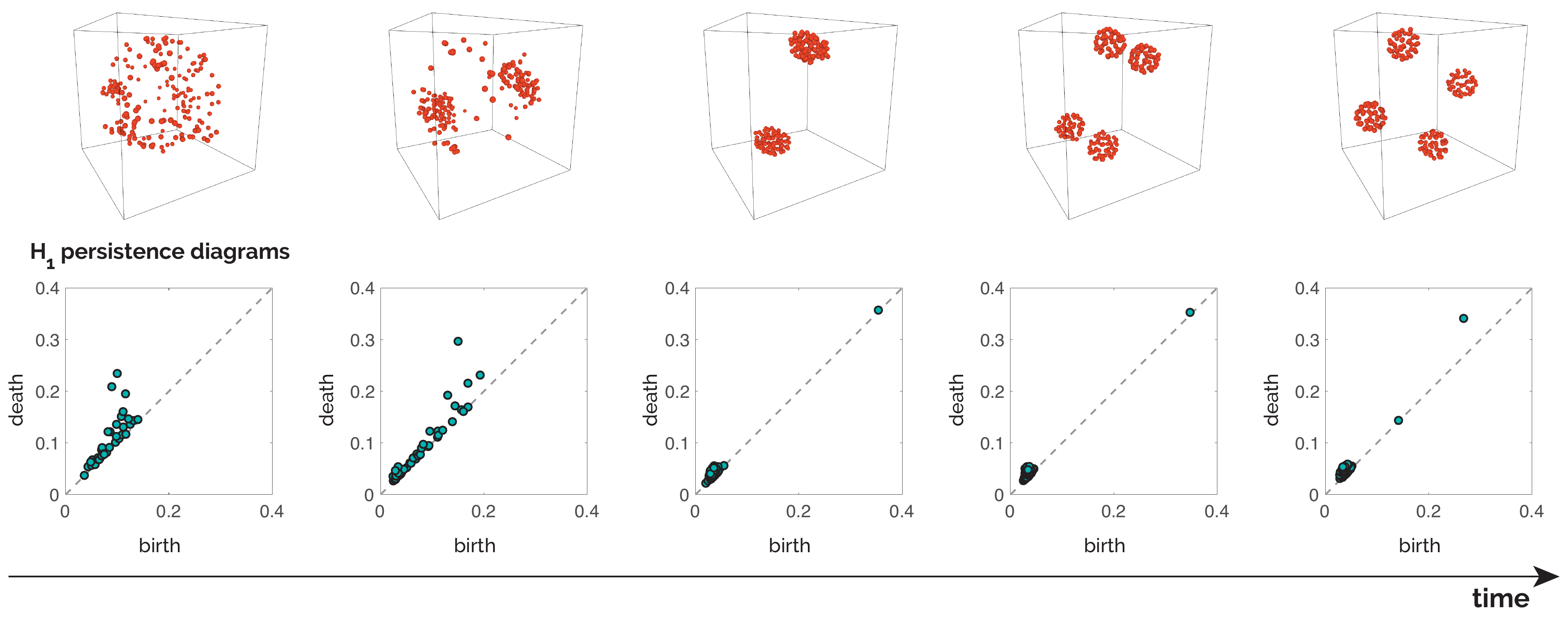}
                \caption{Snapshots from a time-varying point cloud in $\R^3$ (top row) generate corresponding frames along a path of $H_1$ persistence diagrams (bottom row).}
                \label{fig:path_of_pd}
    \end{figure}

One tool for characterizing such organizational structures is  persistent homology~\cite{chazal_structure_2016, edelsbrunner_persistent_2008,ghrist_barcodes:_2008-3, otter_roadmap_2017, zomorodian_computing_2005}, a common tool from topological data analysis (TDA). It is applied by constructing a nested family of spaces, which model the structure of the underlying system at scale parameters $\epsilon$, and summarizes the topological features of this family in an object called a \emph{persistence diagram}. Persistence diagrams provide a multi-scale summary of the system's organization, while discarding precise information about individual elements in a system. Thus, they are well-suited to characterizing the state of a multi-agent complex system.

While the details of individual agents can often be abstracted away when describing a static snapshot of a complex system, the same cannot be said for the temporal structure when characterizing dynamics. Even if the ensemble of states the system attains is similar, the ordering of these states must be carefully recorded to provide understanding. For example, a system in which agents move into consensus versus one in which they move out of consensus are operating under very different dynamic principles, even though reordering the time dimension so one runs backward will result in a similar collection of coarse system organizations, and thus similar topological signatures. \emph{Paths of persistence diagrams}\footnote{Here, we specifically mean paths in the space of persistence diagrams, as compared to the concept of a \emph{persistence vineyard}~\cite{cohen-steiner_vines_2006}, which retains information about how individual homology classes evolve over time. Such granular information is often difficult to obtain in the context of real data, particularly when individual elements of the system cannot be tracked.}~\cite{cohen-steiner_vines_2006, munch_applications_2013, perea_sliding_2015, rieck_uncovering_2020, topaz_topological_2015}, illustrated in~\Cref{fig:path_of_pd}, provide a hybrid form of measurement that retains this temporal\footnote{The methods we will develop work without change for any one-dimensional parameterized family of persistence diagrams, however this terminology collides with that of the scale parameter so we will favor the dynamic terminology throughout.} information while discarding the relatively uninformative data about the individual agents. 

However, as mathematical objects, paths of persistence diagrams are not straightforward to understand and measure.
Neither spaces of persistence diagrams nor spaces of paths are easy to vectorize, nor to apply statistical methods to. To address this problem, in this paper we to develop a collection of informative, stable, and computable features for the space of paths of persistence diagrams, and demonstrate one application of those features in classifying the dynamic behavior of a complex system.

The study of path spaces is a classical topic in algebraic topology~\cite{chen_iterated_1977-1}. Recent work has leveraged and extended these ideas to provide sophisticated feature sets in the form of \emph{path signatures} for time series analysis in machine learning~\cite{chevyrev_primer_2016, kiraly_kernels_2019, adams_persistence_2017, morrill_generalised_2021, lee_path_2020, kidger_deep_2019}. The path signature is a reparametrization-invariant characterization of paths valued in a Banach space, and has been shown to be \emph{universal} and \emph{characteristic}~\cite{chevyrev_signature_2022}, two properties which provide strong theoretical guarantees in the context of kernel methods.

\begin{figure}[t]
    \centering
            \includegraphics[width=0.98\textwidth]{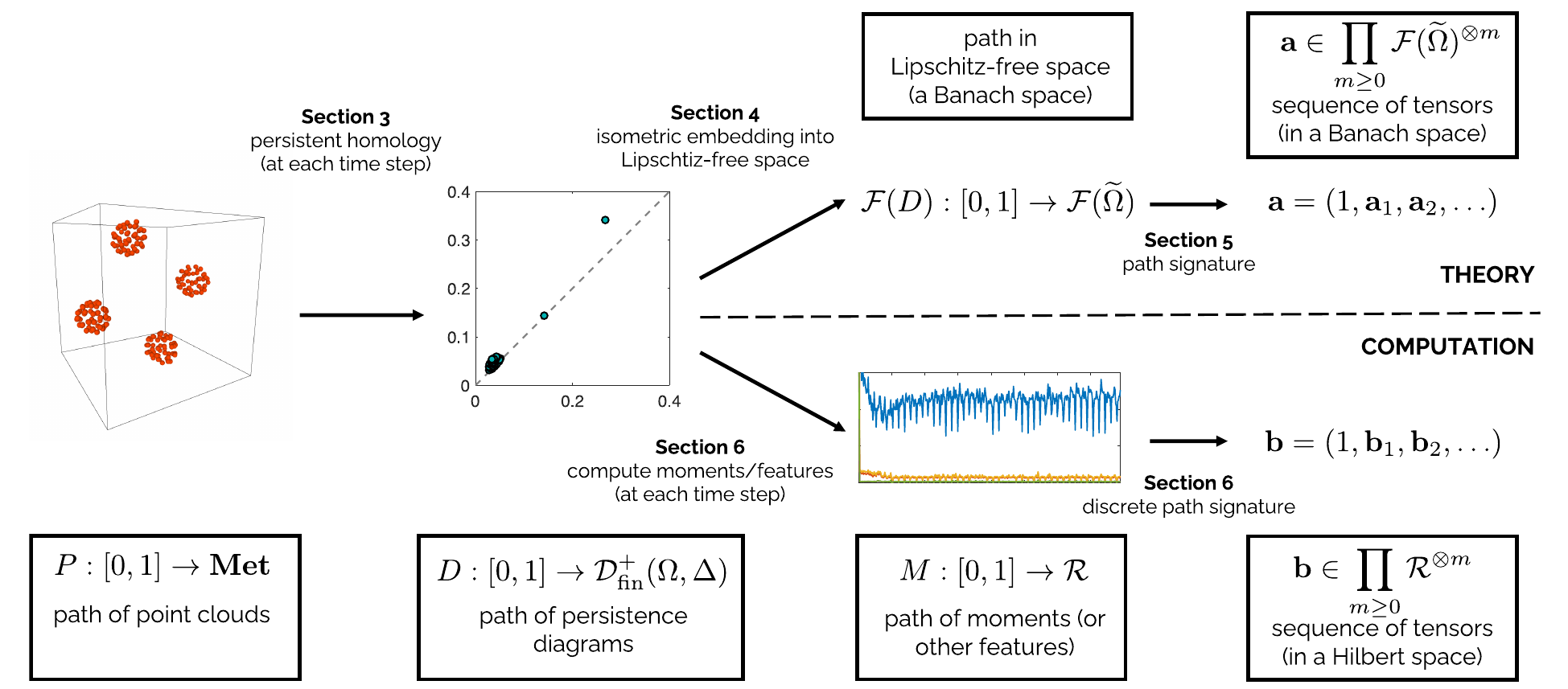}
                \caption{Computational pipeline for studying paths of point cloud, including outline of article.}
                \label{fig:outline}
    \end{figure}

\subsection{Contributions and Outline}

Throughout this article, we will use the problem of quantifying the temporal topological structure of dynamic multi-agent systems as our running example. The structure of this article follows the computational pipeline for the analysis of such data, and is shown in~\Cref{fig:outline}. We begin in~\Cref{sec:feature} with an overview of feature maps and their universal and characteristic properties, forming the foundation of our approach. We then begin the methods in our pipeline, and provide an overview of persistent homology, along with metric properties of the space of persistence diagrams.

The remainder of the article is split into two major parts. The \emph{theoretical} part intrinsically defines path signatures on the space of persistence diagrams. This is done by isometrically embedding the space of persistence diagrams into a Banach space in~\Cref{sec:pd_banach}, relating persistence diagrams with the theory of Lipschtiz-free space~\cite{weaver_lipschitz_2018}. This is followed by defining the signature on this Banach space in~\Cref{sec:pd_paths}.

However, we do not have explicit, finite-dimensional representations of the elements of the Lipschitz-free space, and thus we must first vectorize the static persistence diagrams. In~\Cref{sec:computation}, we introduce \emph{persistence moments}, a novel stable and injective feature map for bounded persistence diagrams, which allows us to summarize persistence diagrams in a low-dimensional representation. We then consider the discrete path signature, which can be efficiently computed in practice, and finally apply this pipeline to a parameter estimation problem for multi-agent systems. 
In particular, our main contributions are as follows:

\begin{itemize}
    \item We connect the theory of persistence diagrams with the theory of Lipschitz-free spaces~\cite{weaver_lipschitz_2018} in order to obtain an isometric embedding of the space of persistence diagrams into a Banach space. 

    \item We introduce the space of bounded variation paths of persistence diagrams and show that the path signature provides a universal and characteristic feature map for this space. In particular, the expected signature characterizes measures on this path space, and can thus be used for statistical applications. In addition, this characterization provides a foundation for the investigation of stochastic dynamics using persistent homology via analogy to the theory of rough paths.
    
    \item We introduce \emph{persistence moments}, a new feature map for persistence diagrams motivated by the perspective of persistence diagrams as measures. Specifically, by computing the moments of bounded persistence diagrams, we obtain an injective stable feature map for the space of persistence diagrams.
    
    \item Combining the moment map with the discrete path signature, we obtain a computable feature map for paths of persistence diagrams. Furthermore, this feature map is stable with respect to the $1$-variation metric on paths, and is also universal and characteristic. 
    
    \item We develop code to perform signature kernel computations for paths of persistence diagrams, which can be found at:
    \begin{center}
        \url{https://github.com/ldarrick/paths-of-persistence-diagrams}.
    \end{center}

\end{itemize}

\begin{remark}
    The results in this work originally appeared in the second author's PhD dissertation. While the authors were preparing this manuscript, the article~\cite{bubenik_virtual_2020} was updated to include an independent discovery of the isometric embedding of the space of persistence diagrams into a Lipschitz-free space. While there is some overlap between the papers, the perspectives are fundamentally different. The results of~\cite{bubenik_virtual_2020} discuss the Lipschitz-free construction via category theory, while our article focuses on the connections between Lipschitz-free spaces and measures, along with applications to path signatures of persistence diagrams.
\end{remark}

\subsection{Previous and Related Work}
The question of how to study dynamic data using persistent homology reaches back to the early development of the subject. Perhaps the earliest initial approach is through persistence vineyards, introduced by Cohen-Steiner et al.~\cite{cohen-steiner_vines_2006}.  While persistence vineyards capture substantial information about the evolution of topological summaries, they require a more complex representation and poses several computational and theoretical challenges. Later, Munch et al.~\cite{munch_probabilistic_2015} studied the paths of persistence diagrams through the lens of statistical and probabilistic features, such as Fr\'echet means. An alternative approach to the study of dynamic metric spaces was considered by Kim and M\'emoli~\cite{kim_spatiotemporal_2020} by building a 3-parameter filtration, and using methods from multiparameter persistence. \medskip

Because their features are often understood qualitatively, complex time-varying systems are a popular potential application of persistent homology. One application of broad interest is the dynamic topological structure of swarming/flocking behavior, which was studied by Topaz et al.~\cite{topaz_topological_2015} and Bhaskar et al.~\cite{bhaskar_analyzing_2019}. For purposes of interpretability and comparison with existing methods, we use simulations of swarm data in our computational experiments. Other examples of time-varying persistence diagrams arise, for example, in neuroscience, where they have been used to study time-varying fMRI ~\cite{rieck_uncovering_2020} and as EEG data~\cite{yoo_topological_2016}. In another direction, persistent homology has been used to detect periodicity in time series through sliding window embeddings~\cite{perea_sliding_2015, perea_sw1pers_2015}. This method has been used to detect chatter in physical cutting processes such as turning and milling~\cite{khasawneh_chatter_2016} and biphonation in videos of vibrating vocal folds~\cite{tralie_quasiperiodicity_2018}.

Path signatures were originally developed as the \emph{iterated integral model for path spaces} by Chen~\cite{chen_integration_1958} to characterize paths on manifolds, which was subsequently generalized to a de Rham-type cochain algebra for loop spaces~\cite{chen_iterated_1977-1}. Path signatures were later used as the foundation for the theory of rough paths by Lyons~\cite{lyons_differential_1998}, and has been transformed into a powerful feature set for time series data~\cite{chevyrev_primer_2016}. Chevyrev and Oberhauser~\cite{chevyrev_signature_2022} showed the signature is universal and characteristic, providing a theoretical justification for its use in the context of kernel methods, facilitated by Kiraly and Oberhauser's~\cite{kiraly_kernels_2019} development of efficient algorithms to compute truncated signature kernels. A more recent method to compute untruncated signature kernels using PDEs was developed by Cass et al.~\cite{cass_computing_2021}. The connection between path signatures and persistent homology was initiated by Chevyrev, Nanda and Oberhauser~\cite{chevyrev_persistence_2020} to develop feature maps for static persistence diagrams. \medskip

Interest in applying statistical and machine learning tools in the context of topological data analysis has also driven recent work on generalized spaces of persistence diagrams. Bubenik and Elchesen~\cite{bubenik_universality_2020, bubenik_virtual_2020} study the space of persistence diagrams from an algebraic and category theoretic perspective, and develop spaces of persistence diagrams with additional structure through the use of universal properties. The connection between persistence diagrams and the theory of optimal transport was made explicit by Divol and Lacombe~\cite{divol_understanding_2021}.

Recent work shows the space of finite persistence diagrams does not admit an isometry into a Hilbert space~\cite{turner_same_2020} for any $p$-Wasserstein metric with $p \in [1, \infty]$. When $p \in (2,\infty]$, there does not even exist a \emph{coarse embedding} into a Hilbert space~\cite{bubenik_embeddings_2019, wagner_nonembeddability_2021}. Finally, the limitations of bi-Lipschitz embeddings into a Hilbert space are discussed in~\cite{carriere_metric_2019}. Given these limitations, an isometric embedding of the space of persistence diagrams into a Banach space is among the strongest positive results one could reasonably hope for. \medskip

\section{Feature Maps}
\label{sec:feature}

In this paper, we study two classes of data, persistence diagrams and paths, which are natively elements of nonlinear spaces. This creates computational difficulties in applying methods from machine learning to their study, as we cannot directly represent elements of these spaces using fixed finite-dimensional vectors: the numbers of points in a persistence diagram may vary, and even discretized time series may have different lengths. To address this problem, a common approach is to construct \emph{feature maps}. Feature maps which are universal and characteristic (see \Cref{def:univ_char}) provide theoretical guarantees for problems related to learning both functions and measures on the input space.\medskip

The feature maps discussed in the theoretical part of this paper are not valued in Hilbert spaces; rather they will be valued in a non-reflexive Banach space. However, the concept of a universal and characteristic feature map still hold in this setting, and we will show that our feature map for paths of persistence diagrams is universal and characteristic in~\Cref{sec:pd_paths}. In~\Cref{sec:computation}, we provide an approximate feature map which is valued in a Hilbert space, and show that it is also universal and characteristic. This allows us to exploit the computational properties of kernel methods, and we apply this in numerical experiments in~\Cref{sec:applications}.

\subsection{Universal and Characteristic Feature Maps}

Given a topological space $\cX$ whose elements are of the  desired input data type, a feature map is a continuous function
\begin{equation}
\label{eq:featuremap}
    \Phi: \cX \rightarrow V
\end{equation}
into a topological vector space $V$. When $V $ is a Hilbert space, we define the \emph{kernel function}
\begin{align}
\label{eq:featuremap_kernel}
    \kappa: \cX \times \cX \rightarrow \R, \quad \kappa(x, x') \coloneqq \langle \Phi(x), \Phi(x') \rangle_V.
\end{align}
The idea behind kernel methods in machine learning is to translate nonlinear learning problems in the input space $\cX$ into linear learning problems in the feature space $V$. By using kernelized algorithms which only depend on the kernel in \Cref{eq:featuremap_kernel}, we bypass the issue of computing potentially infinite dimensional representations of the input data. These methods allow us to study two broad classes or problems in machine learning:

\begin{enumerate}
    \item We can use linear functionals $\langle \ell, \Phi(\cdot) \rangle_V$ to \emph{approximate functions} on $\cX$.

    \item Let $\cM_{\fin}^+(\cX)$ denote the space of finite Radon measures on $\cX$. We can use the \emph{kernel mean embedding} (KME) $ \overline{\Phi}: \cM^+_{\fin}(\cX) \rightarrow V$ of $\mu \in \cM_{\fin}^+(\cX)$ defined by $\overline{\Phi}(\mu) = \int_{\cX} \Phi(x) d\mu(x)$ to \emph{represent measures} as elements of $V$. 
\end{enumerate}

\medskip

\begin{remark}
While the requirement that $V$ is a Hilbert space may not be satisfied in certain situations, recent work has generalized the Hilbert space kernel methods to the setting of Banach spaces~\cite{zhang_reproducing_2009, fukumizu_learning_2011, schlegel_approximate_2020}. Although we do not pursue this approach in the present paper, this suggests a potential avenue for future work.
\end{remark}

In order to effectively study functions and measures, the feature map must satisfy additional properties. Here, $V$ is a topological vector space, $V'$ is the topological dual of continuous linear functionals, and $V^*$ is the algebraic dual of all linear functionals. These duals are equivalent when $V$ is finite dimensional. The following discussion is adapted from~\cite{chevyrev_signature_2022}.

Fix a function class $\cG \subset \R^\cX$. We will require a generalization of the KME defined on the dual $\cG'$, which may include distributions. Rather than defining the KME using an integral, the KME will send a distribution $D \in \cG'$ to a linear functional on $V'$ as
\begin{align}
\label{eq:KME_distributions}
    \overline{\Phi}: \cG' \rightarrow (V')^*, \quad D  \mapsto \big(\ell \mapsto D(\langle \ell, \Phi(\cdot)\rangle_V)\big),
\end{align}
where $\ell \in V'$. With this generalized KME, we can define universal and characteristic maps.

\begin{definition}
\label{def:univ_char}
 Let $\cX$ be a topological space, and $ \cG \subset \R^\cX$. Consider a feature map
 \begin{align*}
     \Phi : \cX \rightarrow V.
 \end{align*}
 Suppose that $\langle \ell, \Phi(\cdot ) \rangle \in \cG$ for all $\ell \in V'$. We say that $\Phi$ is
 \begin{enumerate}
     \item \emph{universal to $\cG$} if the map $\iota : V \rightarrow \cG$ with $\iota(\ell) = \langle \ell, \Phi(\cdot) \rangle$ has a dense image in $\cG$; and

     \item \emph{characteristic to $\cG'$} if the KME map from~\Cref{eq:KME_distributions} is injective.
 \end{enumerate}
\end{definition}

We have the following equivalence between universal and characteristic maps.
\begin{theorem}[\cite{chevyrev_signature_2022}]
\label{thm:duality}
    Suppose that $\cG$ is a locally convex topological vector space. A feature map $\Phi$ is universal to $\cG$ if and only if $\Phi$ is characteristic to $\cG'$. 
\end{theorem}

\begin{remark}
This equivalence may seem to stand in contrast with previous work such as in~\cite{kwitt_statistical_2015}, which states that a characteristic feature map is not necessarily universal. This is due to variations in the definition of characteristicness; for example in~\cite{kwitt_statistical_2015}, a feature map is called characteristic if its corresponding KME is injective on probability measures. An extensive discussion on the various notions of universal and characteristic maps is provided in~\cite{simon-gabriel_kernel_2018}.
\end{remark}

\section{Persistent Homology and Persistence Diagrams}
\label{sec:pd_intro}

\begin{figure}
    \centering
            \includegraphics[width=0.98\textwidth]{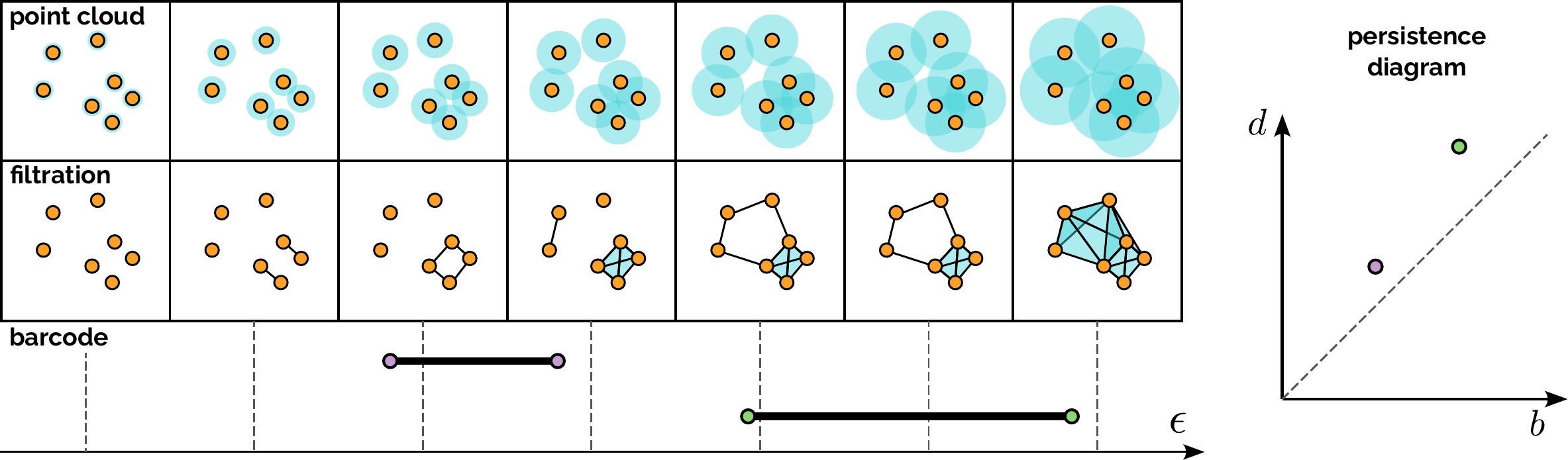}
        \label{fig:persistent_homology}
            \caption{Summary of persistent homology in dimension $1$. The ``point cloud'' row shows a point cloud along with $\epsilon$-balls of increasing radius. The ``filtration'' row shows the corresponding sequence of Vietoris-Rips complexes. The ``barcode'' row shows the birth and death of $1$-dimensional homology classes as a set of bars. Finally, the corresponding persistence diagram is shown on the right.}
\end{figure}

Persistent homology is a tool to capture the multi-scale topological structure of data sets such as point clouds. Intrinsically, point clouds have trivial topological structure as it is simply a set of disjoint points. However, nontrivial topological features may arise if we view the point cloud at different scales. Given a \emph{scale parameter} $\epsilon >0$, we build a topological space from the point cloud by filling in the convex hulls of each collection of points with pairwise distance smaller than $\epsilon$. This provides a topological space which represents the underlying point cloud at the scale parameter $\epsilon$.

Now the natural question arises: how do we choose which $\epsilon$ parameter to use? The key idea behind persistent homology is to consider \emph{all} scale parameters $\epsilon > 0$, and track how the topological features change as the scale parameter is varied. In particular, a \emph{persistence diagram} is an object which summarizes the births and deaths of all topological features of a fixed dimension, which can then be used as a topological summary of a point cloud.

In this section, we provide further details on persistent homology, though we note that there are now many excellent introductions to the subject~\cite{dey_computational_2022-1, chazal_structure_2016, otter_roadmap_2017}. Throughout this article, we will restrict ourselves to the case of point clouds, though we emphasize that persistent homology is a much more general tool which can be applied in many other contexts. Next, we will give an exposition of the main properties of the space of persistence diagrams, as this will be the focus of the following section. 

\subsection{Persistent Homology}
Suppose $\bX = \{\bx_i\}_{i=1}^N$ is a point cloud, where each $\bx_i \in \R^d$. The method described in this section is summarized in~\Cref{fig:persistent_homology}.\medskip

\paragraph{Building the Vietoris-Rips Filtration}
The first step in persistent homology is to build a \emph{filtration}: a sequence of topological spaces which model the point cloud $\bX$ at every scale parameter $\epsilon >0$. There are several different ways to construct these spaces, but we will focus on the most commonly used filtration (which is implemented in most persistent homology packages) called the \emph{Vietoris-Rips filtration}. This consists of a topological space $\VR_\epsilon(\bX)$ for each $\epsilon > 0$, along with inclusion maps
\[
    \VR_\epsilon(\bX) \hookrightarrow \VR_{\epsilon'}(\bX)
\]  
for any $\epsilon < \epsilon'$. 

Recall that a \emph{simplicial complex} can be represented as a subset $S \subset 2^{\bX}$ of the power set of points in $\bX$, which is closed under taking subsets: if $\sigma \in S$ and $\tau \subset \sigma$, then $\tau \in S$. Each subset $\sigma \in S$ corresponds to a \emph{simplex}:
\begin{itemize}
    \item a subset $\sigma = \{\bx_0\}$ is a \emph{0-simplex} representing the point $\bx_0$;
    \item a subset $\sigma = \{\bx_0, \bx_1\}$ is a \emph{1-simplex} representing the edge which connects $\bx_0$ with $\bx_1$;
    \item a subset $\sigma = \{\bx_0, \bx_1, \bx_2\}$ is a \emph{2-simplex} representing the triangle with vertices given in $\sigma$; and
    \item a subset $\sigma = \{\bx_0, \ldots, \bx_n\}$ is an \emph{n-simplex} representing the n-dimensional simplex with vertices given in $\sigma$. 
\end{itemize}

The \emph{Vietoris Rips complex} at parameter $\epsilon > 0$ is defined to be 
\[
    \VR_\epsilon(\bX) \coloneqq \big\{ \sigma = \{\bx_{i_0}, \ldots, \bx_{i_n}\} \, : \, n < N, \, i_0 < \ldots < i_n, \, d(\bx_{i_j}, \bx_{i_k}) < \epsilon\big\}. 
\]
By definition, $\VR_\epsilon(\bX) \subset \VR_{\epsilon'}(\bX)$ if $\epsilon < \epsilon'$, so there exist inclusion maps $\VR_\epsilon(\bX) \hookrightarrow \VR_{\epsilon'}(\bX)$. \medskip

\paragraph{Persistence Modules}
Now that we are equipped with a filtration of topological spaces, we can compute the simplicial homology of these spaces for a fixed dimension $k$,
\[
    \PH_{k,\epsilon}(\bX) \coloneqq H_k(\VR_\epsilon(\bX), \Z_2).
\]
We use homology with $\Z_2$ field coefficients, since this is the most efficient for computation, and thus the homology groups are vector spaces.  For $\epsilon < \epsilon'$, we also obtain linear maps
\[
    \rho_{\epsilon, \epsilon'}: \PH_{k,\epsilon}(\bX) \rightarrow \PH_{k,\epsilon'}(\bX),
\]
due to the existence of inclusion maps between the corresponding Vietoris-Rips complexes and the functoriality of homology. \medskip

\paragraph{Persistence Diagrams}
Through the linear maps in the previous section, the persistence module contains information about when a homology class is ``born'' (the first $\epsilon$ at which the homology class appears), and when the same homology class ``dies'' (the last $\epsilon$ at which the same homology class is nontrivial). In fact, due to the \emph{decomposition theorem}~\cite{chazal_structure_2016}, this is precisely all of the information contained in the persistence module. A \emph{persistence diagram} collects the birth and death parameters of all homology generators across all scales.

More formally, we say that a homology class $\sigma \in \PH_\epsilon(\bX)$ is \emph{born} at parameter $b_\sigma>0$ if $\sigma \in \im(\rho_{b_\sigma, \epsilon})$ but $\sigma \notin \im(\rho_{b-\delta, \epsilon})$ for any $\delta >0$. Similarly, the homology class $\sigma$ \emph{dies} at parameter $d_\sigma>0$ if $\rho_{\epsilon, d_\sigma}(\sigma) \neq 0$, but $\rho_{\epsilon, d_\sigma+\delta}(\sigma) = 0$ for all $\delta >0$. 

The \emph{persistence diagram} with respect to $\bX$ is the multi-set (repetitions of elements can exist) of birth-death pairs
\[
    \PD_k(\bX) = \{ (b_\sigma, d_\sigma) \, : \, \sigma \text{ is a $k$-dimensional persistent homology generator}\}.
\]
The general idea is that $\sigma$ ranges over all ``independent'' homology classes over the scale parameter, where classes at different $\epsilon$ can be identified via the $\rho$ map. For a more rigorous definition and further details, see~\cite{dey_computational_2022-1}. Furthermore, note that this procedure produces a \emph{finite} persistence diagram (the multiset contains finitely many elements), since we started with a finite point cloud.

\subsection{The Metric Space of Persistence Diagrams and Measures} \label{ssec:metric_pers}
The resulting persistence diagram $\PD_k(\bX)$ represents a summary of the $k$-dimensional topological features in $\bX$, which we aim to use for problems in data science such as classification or regression. However, we require some additional structure on the space of persistence diagrams for this to be a useful tool in data science. Here, we discuss two such properties.

\begin{enumerate}
    \item The \emph{partial $p$-Wasserstein metrics} on the space of diagrams.
    \item The generalization of diagrams to \emph{persistence measures}, which provides a notion of scaling, and allows us to perform normalization and compute statistics such as mean. 
\end{enumerate}\medskip

\paragraph{The Space of Finite Persistence Diagrams}
We begin by formally defining the space of persistence diagrams, following the perspective introduced in~\cite{bubenik_universality_2020,bubenik_virtual_2020}, which can be easily generalized to other settings. A persistence diagram is a multiset of points in $\R^2$ supported in the half-plane above the diagonal. More formally, let
\begin{align}\label{eq:pers_omega_delta}
    \Omega \coloneqq \{(b,d) \in \R^2 \, : \, b \leq d\}, \quad \Delta \coloneqq \{(x,x) \in \R^2\} \subset \Omega.
\end{align}
Given a set $Z$, we let $D^+(Z)$ denote the free commutative monoid on $Z$, consisting of formal sums $\sum_{i=1}^n n_i z_i$, where $n_i \in \N$ and $z_i \in Z$. The \emph{space of finite persistence diagrams} is
\[
    \cD^+_{\fin}(\Omega, \Delta) \coloneqq D^+(\Omega)/D^+(\Delta) \cong D^+(\Omega - \Delta),
\]
which we simplify to $\cD^+_{\fin}$ in this section.\medskip

\paragraph{Partial Wasserstein Metrics on Diagrams}
The most commonly used metrics for persistence diagrams are the partial Wasserstein metrics\footnote{While this is often called the Wasserstein metric in the topological data analysis literature, the diagram metrics do not correspond to the usual notions of Wasserstein metrics in optimal transport theory. Instead, it corresponds to the partial Wasserstein metrics~\cite{figalli_new_2010}, as discussed in~\cite{divol_understanding_2021}.}.
The standard definition is given by partial matchings. Let $D = \{z_i\}_{i=1}^{n}, D' = \{z'_i\}_{i=1}^{n'} \in \cD^+_{\fin}$ be two persistence diagrams, and $d$ be the Euclidean metric on $\Omega \subset \R^2$. A \emph{partial matching} between $D$ and $D'$ is a partially defined injective map $\Gamma: D \dashrightarrow D'$. Let $M_\Gamma \subset D$ denote the subset of $D$ on which $\Gamma$ is defined and $U_{\Gamma} \subset D$ and $U^{\Gamma} \subset D'$ be the unmatched points in $D$ and $D'$ respectively. For $p \in [1,\infty)$, the \emph{partial $p$-Wasserstein metric} $W^\partial_p[d]: \cD^+_\fin \times \cD^+_\fin \to \R$ is defined to be
\[
    W^\partial_p[d](D, D') \coloneqq \inf_{\Gamma} \left( \sum_{z \in M_{\Gamma}} d(z, \Gamma(z))^p + \sum_{z \in U_{\Gamma}} d(z, \Delta)^p + \sum_{z' \in U^{\Gamma}} d(z', \Delta)^p\right)^{1/p},
\]
where the infimum is taken over all partial matchings and $d(z, \Delta)$ denotes the minimum distance between $z$ and the subset $\Delta \subset \Omega$. 

However, we can equivalently define the Wasserstein metric in terms of couplings, which will lead to the generalization to measures. A \emph{coupling} between $D$ and $D'$ is an element $\sigma \in D^+(\Omega \times \Omega)$ such that
\[
    (\pi_1)_* \sigma = D \pmod{D^+(\Delta)} \quad \text{and} \quad (\pi_2)_*\sigma = D' \pmod{D^+(\Delta)},
\]
where $\pi_i : \Omega \times \Omega \to \Omega$ denotes the projection maps to the first and second coordinate.
A coupling is a multiset with the form $\sigma = \{(\sigma_i, \sigma'_i)\}_{i=1}^n$, where $\sigma_i, \sigma_i' \in \Omega$.
With this notion, we can equivalently define the partial $p$-Wasserstein distance by
\[
    W^\partial_p[d](D, D') = \inf_{\sigma} \left(\sum_{i=1}^n d(\sigma_i, \sigma_i')^p\right)^{1/p},
\]
where the infimum is taken over all couplings between $D$ and $D'$.
A matching $\Gamma: D \dashrightarrow D'$ corresponds to a coupling
\[
    \sigma_\Gamma \coloneqq \{(z, \Gamma(z)) \, : \, z \in M_{\Gamma}\} \cup \{(z, \pi_\Delta(z)) \, : \, z \in U_{\Gamma}\} \cup \{(\pi_{\Delta}(z'), z') \, : \, z' \in U^\Gamma\},
\]
where $\pi_\Delta: \Omega \to \Delta \subset \Omega$ is the orthogonal projection onto the diagonal. \medskip

\paragraph{The Space of Persistence Measures}
Often, it is helpful to consider a larger space of objects, which includes the space of persistence diagrams, but also contains other objects of interest (such as the average persistence diagram over some collection), which are not contained in $\cD^+_\fin$. In particular, we introduce a generalization of persistence diagrams where the points on the diagram are weighted, to allow for a notion of scaling. A natural way to achieve this is to view a persistence diagram as a measure.

We can introduce this notion of persistence measures in the same way as for diagrams. Given a topological space $Z$, we let $M^+(Z)$ denote the space of finite (positive) Radon measures on $M^+(Z)$. The \emph{space of finite persistence measures} is defined to be
\[
    \cM^+_{\fin}(\Omega, \Delta) \coloneqq M^+(\Omega)/ M^+(\Delta) \cong M^+(\Omega - \Delta). 
\]
Once again, we often use $\cM^+_\fin$ to simplify notation.\medskip

\paragraph{Partial Wasserstein Metrics on Measures}
The definition of the partial Wasserstein distance using couplings can be immediately generalized to the setting of measures. Let $\mu, \mu' \in \cM^+_\fin$. A \emph{coupling} between $\mu$ and $\mu'$ is a measure $\sigma \in M^+(\Omega \times \Omega)$ such that
\[
    (\pi_1)_*\sigma = \mu \pmod{M^+(\Delta)} \quad \text{and} \quad (\pi_2)_*\sigma = \mu' \pmod{M^+(\Delta)}.
\]
The \emph{partial $p$-Wasserstein metric for measures} $W^\partial_p[d]: \cM^+_\fin \times \cM^+_\fin \to [0, \infty]$ is defined by
\begin{align}
    W_p^\partial[d](\mu, \mu') = \inf_{\sigma} \left( \int_{X \times X} d(x,y)^p d\sigma(x,y)\right)^{1/p},
\end{align}
where the infimum is taken over all couplings between $\mu$ and $\mu'$. However, note that this distance can be possibly infinite since measures may be unbounded. In order to rectify this, we define the \emph{total $p$-persistence} of $\mu$ as
\[
    \Pers_p(\mu) \coloneqq W_p^\partial(\mu, 0),
\]
where $0$ is the zero measure, and define the \emph{space of finite $p$-persistence measures} to be
\[
    \cM^+_{\fin, p}(\Omega, \Delta) \coloneqq \{ \mu \in \cM^+_{\fin}(\Omega, \Delta) \, : \, \Pers_p(\mu) < \infty\}.
\]
The function $W^\partial_p[d]$ restricted to $\cM^+_{\fin, p}$ is an honest metric. Finally, we note that the metric space of persistence diagrams isometrically embeds into this space of persistence measures,
\[
    (\cD^+_{\fin}(\Omega, \Delta), W^\partial_p[d]) \hookrightarrow (\cM^+_{\fin, p}(\Omega, \Delta), W^\partial_p[d]),  
\]
with the embedding given by
\[
    \{z_i\}_{i=1}^n \mapsto \sum_{i=1}^n \delta_{z_i},
\]
where $\delta_z$ is the Dirac measure at $z \in \Omega$. 

\section{A Banach Space for Persistence Diagrams}
\label{sec:pd_banach}
The space of persistence measures already provides significantly more structure than the space of persistence diagrams; however, we still require additional structure in order to define metrics on paths of peristence diagrams and the path signature. While addition is well defined for positive persistence measures, we require additive inverses and completeness in order to define Stieltjes-type integrals and limits. Furthermore, the metric should be compatible with this linear structure. In other words, we need to further enrich the space of persistence diagrams into a Banach space. 
We briefly summarize several \emph{enrichments} of $\cD_{\fin}^+$ which satisfy some of the desirable properties described above; these are depicted in Figure \ref{fig:venn}.

\begin{figure}
\centering
		\includegraphics[width=0.4\textwidth]{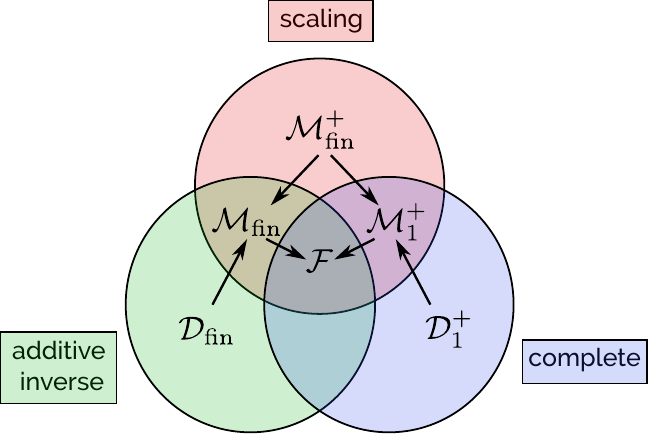}
	\label{fig:venn}
		\caption{Enrichments of the space  $\cD_{\fin}^+$ of finite persistence diagrams and the properties of each; arrows indicate inclusion. Note that $\cD_{\fin}^+$ includes into all of these enriched spaces.}
\end{figure}

\begin{enumerate}
    \item \textbf{Scaling:} The perspective of persistence diagrams as measures was introduced in~\cite{chazal_structure_2016} through counting measures. The extension to Radon measures $\cM_{\fin}^+$ was initiated in~\cite{divol_understanding_2021}, which allows for a notion of scaling.
    
    \item \textbf{Completeness:} The finiteness assumption on persistence diagrams introduces an obstruction to being a complete metric space. A resolution to this problem was developed in~\cite{mileyko_probability_2011} by considering diagrams with countably many points, but with a finite $1$-Wasserstein distance to the empty diagram, denoted by $\cD^+_1$. A similar condition was studied for Radon measures with infinite mass in~\cite{divol_understanding_2021}, denoted $\cM_1^+$.
    
    \item \textbf{Additive Inverse:} Additive inverses were considered in~\cite{bubenik_virtual_2020} using group completion. The group completion of $\cD_\fin^+$ is the space of diagrams with signed coefficients on points, called \emph{virtual persistence diagrams} and denoted $\cD_\fin$. By taking the group completion of finite Radon measures, we obtain signed finite Radon measures, denoted $\cM_{\fin, 1}$. 
\end{enumerate}

The center of Figure \ref{fig:venn} is the \emph{Lipschitz-free space} of a metric space~\cite{weaver_lipschitz_2018}; the main result of this section shows that there is an isometric embedding of the space of finite persistence diagrams into this Banach space. We begin by introducing the notion of metric pairs from~\cite{bubenik_universality_2020,bubenik_virtual_2020} and their associated quotient  spaces, followed by the discussion on Lipschitz-free spaces.

\subsection{Metric Pairs and Quotient Metric Spaces}
The spaces of persistence diagrams and their enrichments can be framed in a more general setting than the upper half plane. We define a \emph{metric pair} $(\generalspace,d,A)$ to be a locally compact Polish emtric space $(\generalspace,d)$, along with a closed subset $A \subset \generalspace$. We use the symbols $(\generalspace,d,A)$ to denote an arbitrary metric pair, and reserve the symbols $(\Omega, d,\Delta)$ defined in~\Cref{eq:pers_omega_delta} to be the specific pair corresponding to ordinary persistence diagrams, where $d$ is understood to be the Euclidean metric. Note that all of the constructions in~\Cref{ssec:metric_pers} can be formulated for arbitrary metric pairs. 

In light of the previous discussion on persistence diagrams and measures, the subset $A$ is treated as a region which does not contribute to the partial Wasserstein distance. Thus, it is convenient to quotient out the subset $A$ in a metric pair to obtain a pointed metric space.

\begin{definition}
    Let $(\generalspace,d,A)$ be a metric pair. We define the \emph{quotient metric space} $(\qgeneralspace, \tilde{d}, *)$ as a pointed metric space as follows. Let $\qgeneralspace \coloneqq (\generalspace-A) \cup \{*\}$ denote the quotient set obtained by collapsing $A$ to a point $*$, which we treat as the basepoint. The metric $\tilde{d}$ on $\qgeneralspace$ is defined by
    \begin{align*}
        \tilde{d}(x,y) \coloneqq \min \left\{ d(x,y), d(x,A) + d(y,A)\right\},
    \end{align*}
    where $d(x,A) \coloneqq \inf_{y \in A} \{ d(x,y)\}$. By abuse of notation, we view the metric $d$ as a function on the $\qgeneralspace \times \qgeneralspace$ given by $(x,y) \mapsto d(x,y)$ for $x,y \in \generalspace - A$, $(x,*) \mapsto d(x,A)$, and $(*, *) \mapsto 0$. 
\end{definition}

This is shown to be a metric  in~\cite{weaver_lipschitz_2018, bubenik_universality_2020}. Note that pointed metric spaces $(\qgeneralspace, \tilde{d}, *)$ are special cases of metric pairs, and all the definitions above for enriched spaces still hold. When the basepoint is clear, we will suppress it from the notation for the enriched spaces of diagrams. For instance, we set $\cD_{\fin}^+(\qgeneralspace) \coloneqq  \cD_{\fin}^+(\qgeneralspace, *)$. Furthermore, partial Wasserstein distances correspond to ordinary Wasserstein distances on the quotient space; details can be found in Appendix~\ref{apx:quotient_metric}.

\subsection{Lipschitz-Free Spaces}

We will now discuss the construction of a Lipschitz-free space of a metric space~\cite{weaver_lipschitz_2018}, which is a Banach space whose linear structure coincides with the metric structure of the underlying space. We view this as a generalization of the space of persistence measures. The material in this section is primarily adapted\footnote{We keep the notation in this section consistent with the previous section, but this differs slightly from the standard notation for Lipschitz-free spaces. Lipschitz-free spaces are called Arens-Eells spaces in~\cite{weaver_lipschitz_2018}.} from~\cite{weaver_lipschitz_2018, cuth_structure_2016}.\medskip
 
Let $(\generalspace,d,*)$ be a pointed metric space, and let $\cV_{\fin}(\generalspace)$ denote the free vector space on $\generalspace -*$. Given an element $\alpha \in \cV_\fin(\generalspace)$, we can decompose it as $\alpha = \alpha^+ - \alpha^-$, where $\alpha^+$ and $\alpha^-$ are both positive. Note that we can view $\alpha^+$ and $\alpha^-$ as positive measures in $\cM^+(\generalspace)$, and thus equip $\cV_\fin(\generalspace)$ with a norm given by $\|\alpha\|_1 \coloneqq W_1^\partial[d](\alpha^+, \alpha^-)$.

\begin{definition}
    Let $(\generalspace, d, *)$ be a pointed metric space. The \emph{Lipschitz-free space} of $\generalspace$, denoted $\cF(\generalspace)$, is the completion of $\cV_\fin(\generalspace,*)$ with respect to the $1$-Wasserstein norm. 
\end{definition}

As the name implies, the Lipschitz-free space is related to the space of Lipschitz functions. Given a function $f: \generalspace \rightarrow \R$, its \emph{Lipschitz constant} is defined to be
\begin{equation}
    \|f\|_{Lip} \coloneqq \sup_{x \neq y \in \generalspace} \left\{ \frac{|f(x) - f(y)|}{d(x,y)} \right\} \in [0, \infty].
\end{equation}
We define the space of Lipschitz functions on $(\generalspace, d, *)$ to be
\begin{equation}
    \Lip_0(\generalspace) \coloneqq \{ f: \generalspace \rightarrow \R \, : \, \|f\|_{Lip}< \infty, \,f(*) = 0 \}
\end{equation}
where we treat the Lipschitz constant as a norm on $\Lip_0(X)$. Furthermore, $\Lip_0(\generalspace)$ is a Banach space with respect to this Lipschitz norm.\medskip

\paragraph{Universal Property}
The Lipschitz-free space are characterized by a universal property.
\begin{theorem}[\cite{weaver_lipschitz_2018}]
\label{thm:lipschitz_free}
 Let $(\generalspace, d, *)$ be a pointed metric space. Then the map $\iota: \generalspace \rightarrow \cF(\generalspace)$ defined by $\iota(x) \coloneqq \delta_x$ isometrically embeds $\generalspace$ in $\cF(\generalspace)$. If $V$ is any Banach space, and $f: \generalspace \rightarrow V$ is a Lipschitz map which preserves the basepoint, there exists a unique bounded linear map $\tilde{f}: \cF(\generalspace) \rightarrow V$ such that $\tilde{f} \circ \iota = f$; in other words, the following diagram commutes
 \[
 \begin{tikzcd}
    \generalspace \ar["f"]{r} \ar["\iota", swap]{d} & V \\
 \cF(\generalspace) \ar[dashed, "\tilde{f}", swap]{ru}
 \end{tikzcd}.
 \]
 Furthermore, $\|\tilde{f}\|_{Lip} = \|f\|_{Lip}$. 
\end{theorem}

The unique extension is defined by first extending $f$ to $\overline{f}: \cV_\fin(\generalspace) \rightarrow V$, defined as follows. Any $\alpha \in \cV_\fin(\generalspace)$ can be written as $\alpha = \sum_{i=1}^n a_i \delta_{z_i}$, where the $z_i \in \generalspace - \{*\}$ are distinct. Then, 
\begin{align}
\overline{f}(\alpha) \coloneqq \sum_{i=1}^n a_i f(z_i).
\end{align}
This is an extension to the universal properties discussed for $\cD_\fin^+(\generalspace)$ and $\cD_\fin(\generalspace)$ in~\cite{bubenik_universality_2020, bubenik_virtual_2020}. Furthermore, this provides a simple method to prove the stability of maps from $\cF(\generalspace)$ into a Banach space, extending the characterization of Lipschitz maps on $\cM^+_1(\generalspace,A)$  in~\cite{divol_understanding_2021}. \medskip

\paragraph{Duality}

There is an equivalent definition of the $1$-Wasserstein norm by viewing $\cV_\fin(\generalspace)$ as a subspace of $\Lip_0(\generalspace)^*$ and using the dual norm; this is often called \emph{Kantorovich-Rubinstein duality}. Given $f \in \Lip_0(\generalspace)$, and $\alpha \in \cV_\fin(X)$, the pairing is defined to be $\langle \alpha, f \rangle = f(\alpha)$. 

\begin{lemma}
    For any $\alpha \in \cV_\fin(\generalspace)$, we have
    \begin{align*}
        \|\alpha\|_1  = \sup_{f \in \Lip_0(\generalspace)} \frac{\langle \alpha, f \rangle}{\|f\|_{Lip}} .
    \end{align*}
    Furthermore, the dual norm uniquely extends to a norm on $\cF(\generalspace)$ which coincides with $\|\cdot\|_{1}$. 
\end{lemma}

Now that $\cF(\generalspace)$ is realized as a Banach space situated within $\Lip_0(\generalspace)^*$, a natural question is whether $\cF(X)$ is simply the entire dual space. In fact, $\cF(\generalspace)$ is the \emph{predual} of $\Lip_0(\generalspace)$, such that $\cF(\generalspace)^* \cong \Lip_0(\generalspace)$~\cite{weaver_lipschitz_2018}. Furthermore, if $\generalspace$ is infinite, then $\cF(\generalspace)$ is not reflexive~\cite{weaver_lipschitz_2018} and is thus strictly smaller than the dual space of $\Lip_0(\generalspace)$. \medskip

\paragraph{Isometric Embeddings}
Our main result in this section relates the Lipschitz-free space back to persistence diagrams, providing an isometric embedding of $\cD_\fin^+(\generalspace)$ into a Banach space.

\begin{theorem}
    Let $(\generalspace, d, *)$ be a pointed metric space. The inclusion $\cD_\fin^+(\generalspace) \hookrightarrow \cF(\generalspace)$ is an isometric embedding.
\end{theorem}
\begin{proof}
    The space of finite persistence diagrams on $(\generalspace, d, *)$, $\cD_\fin^+(\generalspace)$, isometrically embeds into $\cV_\fin(\generalspace)$. Then, since $\cF(\generalspace)$ is the completion of $\cV_\fin(\generalspace)$, the inclusion $\cD_\fin^+(\generalspace) \hookrightarrow \cF(\generalspace)$ is an isometric embedding.
\end{proof}

Furthermore, the following result clarifies that $\cF(\generalspace)$ is also the completion of finite signed measures. Let $\cM_{\fin}(\generalspace)$ denote the space of finite Radon measures on $\generalspace-*$. Given $\mu \in \cM_{\fin}(\generalspace)$, we can decompose it into finite positive measures as $\mu = \mu^+ - \mu^-$. We can then equip $\cM_{\fin}(\generalspace)$ with the (possibly infinite) norm $\|\mu\|_1 \coloneqq W^\partial_1[d](\mu^+, \mu^-)$, and denote $\cM_{\fin,1}(\generalspace)$ to be the restriction to finite norm elements. 

\begin{proposition}[\cite{ostrovska_generalized_2019}] \label{prop:measures_embedded_lfspace}
    Let $(X, d, *)$ be a Polish pointed metric space. Then $(\cV_\fin(X), \|\cdot\|_1)$ is dense in $(\cM_{\fin, 1}(X), \|\cdot\|_{1})$. Thus, $\cF(X)$ is also the completion of $(\cM_{\fin, 1}(X), \|\cdot\|_{1})$.
\end{proposition}

Therefore, the Lipschitz-free space $\cF(X)$ generalizes both persistence diagrams and measures, and provides the desired Banach space for persistence diagrams. \medskip

\paragraph{Discussion on Explicit Representations}
One of the main limitations of the Lipschitz-free space is the lack of an explicit representation of objects. Indeed, there are only a few explicit constructions of such spaces: the Lipschitz-free space of $\R$ is isomorphic to $L^1(\R$)~\cite{weaver_lipschitz_2018}, and the Lipschitz-free space for convex subsets $C \subset \R^N$ is isometric to a quotient of $L^1(C, \R^N)$~\cite{cuth_isometric_2017}.

The problem of representing elements of $\cF(\generalspace)$ as measures is studied in~\cite{aliaga_integral_2020}, where it is shown that every positive element $\alpha \in \cF(\generalspace)$, where $\langle \alpha, f\rangle \geq 0$ for all $f \in \Lip_0(X)$ such that $f \geq 0$, can be represented as a measure on $\generalspace$. Thus, every \emph{majorizable} element $\alpha \in \cF(\generalspace)$, where $\alpha = \alpha^+ - \alpha^-$ is written as the difference of two positive elements $\alpha^+, \alpha^- \in \cF(\generalspace)$, can be represented as the formal difference of two measures. Another approach to the representation of $\cF(\generalspace)$ in terms of measures is through finite measures on $\generalspace^2_\Delta = \generalspace^2 - \Delta_\generalspace$, where $\Delta_\generalspace$ is the diagonal in $\generalspace^2$. In fact, it is shown in Theorem 2.37 of~\cite{weaver_lipschitz_2018} that $\cF(\generalspace)$ is identified with a quotient of $\ell^1(\generalspace^2_\Delta)$, the space of finite measures on $\generalspace^2_\Delta$ with countable support.

While the lack of an explicit unique representation of Lipschitz-free spaces in terms of measures makes its application to computational problems more challenging, we can use this structure to define the path signature on bounded variation paths of persistence diagrams in the next section. We adapt the signature method for computational purposes by using feature maps in~\Cref{sec:computation} and~\Cref{sec:applications}.

\section{Paths of Persistence Diagrams and the Path Signature}
\label{sec:pd_paths}

Now we turn to the primary objects of interest in this paper: paths of persistence diagrams\footnote{A related object is the \emph{persistence vineyard}, ~\cite{cohen-steiner_vines_2006}. Vineyards summarize how individual persistent homology classes evolve over time, while our methods treat persistence diagrams independently at each time step, and thus contain strictly less information.}. As such, we will fix our metric pair to be $(\Omega, d, \Delta)$ as defined in~\Cref{eq:pers_omega_delta}, and let $(\wOmega, \wid, *)$ be the corresponding quotient space. Such paths naturally arise in topological data analysis by applying persistent homology to dynamic data sets at each time point, and are thus straightforward to compute.  \medskip

In this section, we introduce the Banach space of bounded variation paths of persistence diagrams, which provides sufficient structure to study statistics such as the expectation of probability measures on this space. Next, we introduce the \emph{path signature}, a reparametrization invariant characterization of paths valued in a Banach space, and show that it is universal and characteristic. Due to the lack of explicit representations of Lipschitz-free spaces, such signatures are infeasible to compute in practice. However, this still provides a theoretical framework which can potentially be extended to study the dynamic topological structure of highly irregular stochastic systems.

\subsection{Bounded Variation Paths of Persistence Diagrams}
\begin{definition}
    Suppose $(V, \|\cdot\|)$ is a Banach space. Let $\gamma: [0,1] \rightarrow V$. The \emph{$p$-variation of $\gamma$} is defined to be
    \begin{equation}
        |\gamma|_{p-var} \coloneqq  \sup_{\Pi} \left(\sum_{i=1}^n \|\gamma(t_{i})- \gamma(t_{i-1})\|^p \right)^{1/p}
    \end{equation}
    where $\Pi = \{0 = t_0 < t_1 < \ldots < t_{n-1} < t_n = 1 \}$ is a partition of the interval $[0,1]$, and the supremum is taken over all possible partitions. The space of all \emph{bounded $p$-variation paths} is denoted by
    \begin{equation}
        C^{p-var}([0,1], V) \coloneqq \{ \gamma: [0,1] \rightarrow V \, : \, |\gamma|_{p-var} < \infty\}.
    \end{equation}
    Furthermore, if we require all paths to begin at a specified base point $b\in V$, we will denote the space by $C^{1-var}_b([0,1], V)$. Unless otherwise specified, we will use the unit interval as the domain, and thus simplify the notation as
    \begin{align*}
        C^{1-var}(V) \coloneqq C^{1-var}([0,1], V), \quad \quad C^{1-var}_b(V) \coloneqq C^{1-var}_b([0,1], V).
    \end{align*}
\end{definition}

\begin{proposition}[\cite{chistyakov_maps_1998}]
   Let $(V, \|\cdot\|)$ be a Banach space and $p \geq 1$. Then $C^{p-var}([0,1], V)$, equipped with the norm $\|\gamma\|_p \coloneqq |\gamma(0)| + |\gamma|_{p-var}$, is a Banach space.
\end{proposition}

For paths of bounded $1$-variation, we will simply call them paths of bounded variation. In particular, we will be working with the Banach space of bounded variation paths of persistence diagrams, $C^{1-var}(\cF(\widetilde{\Omega}))$.

\subsection{The Path Signature}
The \emph{path signature} is a powerful reparametrization-invariant characterization for multivariate (and possibly infinite-dimensional) time series. We use the path signature to characterize paths of persistence diagrams as elements of $\cF(\widetilde{\Omega})$. Thus, we require the theory of signatures for paths valued in infinite dimensional Banach spaces, for which we refer to~\cite{lyons_system_2007, boedihardjo_note_2015}. The finite dimensional theory can be made more explicit, and we provide some brief remarks along these lines to aid exposition. A more thorough introduction to the finite-dimensional theory in machine learning can be found in~\cite{giusti_iterated_2020, chevyrev_primer_2016}. \medskip

The path signature is a map which takes paths to a formal power series of tensors. Let $(V, \|\cdot\|)$ be a Banach space. Formal power series $\bs = (\bs_k)_{k=0}^\infty \in \prod_{k=0}^\infty V^{\otimes k}$ consist of a sequence of elements $\bs_0 = 1 \in \R$ and $\bs_k \in V^{\otimes k}$, $k \geq 1,$ where $V^{\otimes k}$ denotes the completion of the algebraic $k$-fold tensor product of $V$ under an admissible tensor norm\footnote{If $V$ is a Hilbert space, then $V^{\otimes k}$ is the standard tensor product; however, in the case of infinite-dimensional Banach spaces, some subtle technicalities arise, and we refer the interested reader to~\cite{lyons_system_2007,boedihardjo_note_2015} for further details.} $\|\cdot\|_k$. In particular, each $V^{\otimes k}$ is a Banach space. Addition and scalar multiplication are determined degreewise.

We define a norm on $\prod_{k=0}^\infty V^{\otimes k}$ by
\begin{align}
\label{eq:tensorps_norm}
    \|\bs\| = \left(\sum_{k=1}^\infty \|\bs_k\|_k^2\right)^{1/2},
\end{align}
which may be infinite for some terms. Note that if $V$ is a Hilbert space, this norm coincides with the Hilbert space norm. The algebraic structure we are interested in is the following Banach space of finite norm power series
\begin{align*}
    T\ps{V} \coloneqq \left\{ \bs \in \prod_{k=0}^\infty V^{\otimes k} \, : \, \|\bs\| < \infty \right\}.
\end{align*}
We will call this the \emph{tensor algebra}. The \emph{truncated tensor algebra at degree $M$} is
\begin{align*}
    T^{(M)}(V) \coloneqq  \prod_{k=0}^M V^{\otimes k}
\end{align*}
where the finite norm condition is immediately satisfied.

\begin{definition}
    Let $(V, \|\cdot\|)$ be a Banach space. Let $\gamma \in C^{1-var}(V)$ be a bounded variation path, and $m \geq 1$. The \emph{path signature of $\gamma$ at level $m$,} $S_m: C^{1-var}(V) \rightarrow V^{\otimes m},$ is
    \begin{equation}
    \label{eq:path_signature}
        S_m(\gamma) \coloneqq \int_{\Delta^m(0,1)} d\gamma(t_1) \otimes d\gamma(t_2) \otimes \ldots \otimes d\gamma(t_m),
    \end{equation}
    where $\Delta^m(a,b) \coloneqq \{a \leq t_1 < \ldots < t_m \leq b\}$. This is an iterated integral in the Riemann-Stieltjes sense. The \emph{path signature of $\gamma$}, denoted $S: C^{1-var}(V) \rightarrow T\ps{V}$, is defined by
    \begin{align*}
        S(\gamma) \coloneqq (1, S_1(\gamma), S_2(\gamma), \ldots).
    \end{align*}
    The \emph{truncated path signature of $\gamma$ at level $M$}, denoted $S_{\leq M} : C^{1-var}(V) \to T^{(M)}(V)$, is 
    \[
        S_{\leq M}(\gamma) = (1, S_1(\gamma), S_2(\gamma), \ldots, S_M(\gamma)).
    \]
\end{definition}

The truncation of the signature is justified due to its exponential decay.

\begin{lemma}[Lemma 2.1.1~\cite{lyons_differential_1998}]
    Let $\gamma \in C^{1-var}(V)$. Then for $m \geq 1$,
    \begin{align*}
        \|S_m(\gamma)\| \leq \frac{|\gamma|_{1-var}^m}{m!}.
    \end{align*}
\end{lemma}

\begin{remark}
    If $V = \R^N$ is finite dimensional with an ordered basis $(e_1, \ldots, e_N)$, then a path $\gamma \in C^{1-var}(V)$ can be represented in terms of those coordinates $\gamma = (\gamma_1, \ldots, \gamma_N)$. Furthermore, $V^{\otimes k}$ is the standard tensor product of vector spaces, and has a canonical basis $\{e_I\}$, where $I = (i_1, \ldots, i_k)$ is a multi-index with $i_j \in [N] \coloneqq \{1, \ldots, N\}$. Then, we may express the $e_I$ component of the path signature $S(\gamma)$ as
    \begin{align*}
        S^I(\gamma) &= \int_{\Delta^m(0,1)}\gamma'_{i_1}(t_1) \ldots \gamma'_{i_m}(t_m) dt_1 \ldots dt_m.
    \end{align*}
\end{remark}

There are two standard properties of the path signature that follow immediately from the definitions. For $\gamma \in C^{1-var}(V)$, then we have
\begin{itemize}
    \item translation invariance, where $S(\gamma + v) = S(\gamma)$ for any $v \in V$;
    \item reparametrization invariance, where $S(\gamma \circ \phi) = S(\gamma)$, for a non decreasing bijection $\phi:[0,1] \rightarrow [0,1]$.
\end{itemize}
Furthermore, we can describe the kernel of the signature map.

\begin{definition}
    Let $\alpha, \beta \in C^{1-var}(V)$. The concatenated path $\alpha * \beta$ is defined to be
    \begin{align*}
        (\alpha * \beta)(t) & = \left\{
            \begin{array}{cl}
            \alpha(2t) & : t \in [0,1/2) \\
            \alpha(1) - \beta(0) + \beta(2t-1) & : t \in [1/2, 1].
            \end{array}
        \right.
    \end{align*}
    Given a path $\alpha \in C^{1-var}(V)$, its \emph{inverse} $\alpha^{-1} \in C^{1-var}(V)$ is defined by $\alpha^{-1}(t) \coloneqq \alpha(1-t)$.
\end{definition}

\begin{definition}[\cite{hambly_uniqueness_2010}]
    A path $\gamma \in C^{1-var}(V)$ is \emph{tree-like} if there exists a nonnegative real-valued continuous function $h:[0,1] \rightarrow \R_{\geq 0}$ such that $h(0) = h(1) = 0$ and
    \begin{align*}
        \|\gamma(t) - \gamma(s) \| \leq h(s) + h(t) - 2 \inf_{u \in [s,t]} h(u).
    \end{align*}
    Such a function $h$ is called the \emph{height function}. Two paths $\alpha, \beta \in C^{1-var}(V)$ are \emph{tree-like equivalent}, denoted $\alpha \sim_t \beta$, if $\alpha * \beta^{-1}$ is tree-like. 
\end{definition}

\begin{theorem}[\cite{boedihardjo_signature_2016}]
    Let $\alpha, \beta \in C^{1-var}(V)$. Then $S(\alpha) = S(\beta)$ if and only if $\alpha$ and $\beta$ are tree-like equivalent. 
\end{theorem}

We denote the quotient of bounded variation paths by this equivalence relation by
\begin{align}
\label{eq:treelike_equivclass}
    \widetilde{C}^{1-var}(V) \coloneqq C^{1-var}(V) /\sim_t.
\end{align}

Finally, the continuity of the truncated path signature will allow us to provide a stable feature map for paths of persistence diagrams.
\begin{theorem}[\cite{lyons_system_2007}]
\label{thm:signature_stability}
    Let $\alpha, \beta \in C^{1-var}(V)$ and $\ell \geq \max \{ |\alpha|_{1-var}, |\beta|_{1-var}\}$. There exists a constant $C > 0$ such that if $|\alpha - \beta|_{1-var} < \epsilon$, then
    \begin{align*}
        \|S_m(\alpha) - S_m(\beta)\| \leq \epsilon \frac{\ell^{m-1}}{C^{m-1}m!}.
    \end{align*}
    In particular, this implies that
    \begin{align*}
        \|S(\alpha) - S(\beta)\| \leq \frac{\epsilon C}{\ell} \exp\left(\left(\frac{\ell}{\sqrt{2}C}\right)^2\right).
    \end{align*}
\end{theorem}

\subsection{A Characterization for Paths of Persistence Diagrams}
In this section, we discuss a fundamental result from~\cite{chevyrev_signature_2022} which shows that the path signature is universal and characteristic. We apply this to bounded variation paths of persistence diagrams, and obtain a characterization for such objects. \medskip

The perspective in this section is to view the path signature $S: C^{1-var}(V) \rightarrow T\ps{V}$ as a feature map and consider the universal and characteristic property with respect to $\cG = C_b(C^{1-var}(V))$, continuous bounded functions on the space of bounded variation paths. As described in~\Cref{def:univ_char}, we must first ensure that $\langle \ell, \Phi(\cdot)\rangle_V \in \cG$ for all $\ell \in T\ps{V}'$. In particular, we must show that $\langle \ell,\Phi(\cdot)\rangle_V: \cX \rightarrow \R$ is bounded. This will be achieved using tensor normalization. An explicit construction of a tensor normalization is given in~\cite{chevyrev_signature_2022}. A technical point is the strict topology, which is discussed in~\Cref{apxsec:strict}.

\begin{definition}
    A \emph{tensor normalization} is a continuous injective map of the form
    \begin{align*}
        \Lambda : T\ps{V} &\rightarrow \{ \mathbf{t} \in T\ps{V} \, : \, \|\mathbf{t}\| \leq K \} \\
        \mathbf{t} = (\mathbf{t}_k)_{k=0}^\infty & \mapsto (\lambda(\mathbf{t})^k \mathbf{t})_{k=0}^\infty
    \end{align*}
    where $K > 0$ is a constant, $\lambda : T\ps{V} \rightarrow (0, \infty)$ is a function.
\end{definition}

\begin{theorem}[\cite{chevyrev_signature_2022}]
\label{thm:univ_char_general}
Let $\Lambda: T\ps{V} \rightarrow T\ps{V}$ be a tensor normalization. The normalized signature
\begin{align*}
    \Phi: \widetilde{C}^{1-var}(V) \rightarrow T\ps{V}, \quad \Phi = \Lambda \circ S,
\end{align*}
\begin{enumerate}
    \item is a continuous injection from $\widetilde{C}^{1-var}(V)$ into a bounded subset of $T\ps{V}$,
    \item is universal to $\cG \coloneqq C_b(\widetilde{C}^{1-var}(V))$ equipped with the strict topology, and 
    \item is characteristic to the space of finite regular Borel measures on $\widetilde{C}^{1-var}(V)$. 
\end{enumerate}
\end{theorem}

Applying this to the case of $V = \cF(\widetilde{\Omega})$, we obtain a universal and characteristic map for bounded variation paths of persistence diagrams.

\begin{theorem}
    Let $\Lambda: T\ps{\cF(\wOmega)} \rightarrow T\ps{\cF(\wOmega)}$ be a tensor normalization. The normalized signature $\Lambda \circ S: \widetilde{C}^{1-var}(\cF(\widetilde{\Omega})) \rightarrow T\ps{\cF(\wOmega)}$ is injective, universal with respect to $C_b(\widetilde{C}^{1-var}(V))$ with respect to the strict topology, and characteristic with respect to the space of finite Borel measures on $\widetilde{C}^{1-var}(V)$.
\end{theorem}

However, there is an obstruction to using this signature as a feature map in practice: it is not readily computable. Indeed, the Banach space $\cF(\widetilde{\Omega})$ is infinite dimensional, and thus even the first level of the signature will be infinite dimensional. Furthermore, there is no clear basis for $\cF(\widetilde{\Omega})$ which can be effectively used for computation; this is the same problem as vectorizing persistence diagrams. We resolve this in the next section by approximating the signature by using a feature map for $\cF(\widetilde{\Omega})$ as well as truncation of the signature.

\begin{remark}
    The definition of Lipschitz-free spaces holds for any pointed metric space. Let $\sPM_{ft}$ denote the space of persistence modules of finite type (those which are a finite sum of finite interval modules) equipped with the interleaving distance $d_I$~\cite{chazal_structure_2016}. Let $(\widetilde{sPM}_{ft}, d_I, 0_M)$ denote the pointed metric space of observable isomorphism classes of finite type persistence modules, where two modules $U, V \in \sPM_{ft}$ are identified if $d_I(U,V) = 0$, and $0_M$ denotes the zero module ~\cite{chazal_observable_2016}. By considering the Lipschitz-free space $\cF(\widetilde{sPM}_{ft})$, we can consider the path signature for paths of persistence modules. Such a procedure can also be used to define the signature for paths of multiparameter~\cite{lesnick_theory_2015} or generalized~\cite{kim_generalized_2021} persistence modules. 
\end{remark}

While the signature from this section may not be directly amenable to computation, it forms a foundation for the investigation of stochastic time-varying systems using persistent homology. In particular, the results of this section hold in the much more general setting of \emph{rough paths}, an approach to stochastic analysis which allows us to perform integrals of highly irregular paths~\cite{friz_course_2020}. When studying time-varying systems in the presence of irregular noise such as Brownian motion (which has unbounded $p$ variation for $p \leq 2$ a.s.~), the bounded variation hypothesis often does not hold. For example, consider a stochastic point cloud $\bB(t)$ consisting of two independent Brownian motions $B(t)$ and $B'(t)$. At each time point $t$, the reduced $H_0$ persistence diagram has exactly one point $\PD_0(\bB(t)) = \{(0, |B(t)-B'(t)|/2)\}$, which will also be of unbounded variation. 
Thus, rough paths may allow us to describe the persistent homology of stochastic time-varying systems, which we leave to future work. 

\section{Methods for Application to Data}
\label{sec:computation}

The previous section provided the framework to theoretically study bounded variation paths of persistence diagrams by defining the signature of these paths. In this section, we focus on adapting these tools to provide computable methods for data analysis. Here we make two modifications to the space of persistence diagrams:
\begin{enumerate}
    \item We use birth/lifetime coordinates rather than birth/death coordinates. Lifetime (or persistence) coordinates are more interpretable, and simplify the description of partial Wasserstein metrics, as the distance to the boundary is simply the lifetime coordinate. Birth/lifetime coordinates have previously been used in the description of persistence images~\cite{adams_persistence_2017}. 
    \item We use bounded persistence diagrams, which are supported on the metric pair 
    \begin{align}
        \Omega_T \coloneqq [0,T]^2 \quad \text{and} \quad \Delta_T \coloneqq \{ (b, 0) \in \Omega_T \, : \, b \in [0,T]\}.
    \end{align}
    where $T> 0$. We use $(b,p) \in \Omega_T$ to refer to \emph{birth} and \emph{persistence}.
    Furthermore, let $(\wOmega_T, *)$ denote the corresponding quotient metric space.
    In practical applications, this is a mild assumption since we often work with finite metric spaces built from data.
\end{enumerate}

We begin by introducing a novel feature map for static persistence diagrams by computing all moments.
We then recall the discrete path signature~\cite{kiraly_kernels_2019}, discuss computational aspects of the path signature, and discuss its application to time-varying persistence diagrams.

\subsection{Persistence Moments}

Given that a persistence diagram is a collection of points supported on $\wOmega_T$, one of the simplest ways to vectorize a persistence diagram is by computing various statistics of these points, such as the mean, standard deviation, and quantiles~\cite{asaad_persistent_2022,chung_persistence_2022,pun_persistent-homology-based_2018}. In fact, a recent survey of vectorization methods~\cite{ali_survey_2022} has shown that such \emph{persistence statisics} consistently outperform other vectorization methods on various classification tasks.

However, certain persistence statistics such as the mean birth time is not stable, and furthermore, they are certainly not injective on the space of finite bounded persistence diagrams. In this section, we introduce \emph{persistence moments}, a modification of persistence statistics which provides a stable and injective map.

Let $\alpha_1, \alpha_2 \in \N$. Given a measure $\mu \in \cM_{\fin, 1}^+(\wOmega_T)$, the \emph{$(\alpha_1, \alpha_2)$-moment} of $\mu$ is defined by
\[
    \int_{\wOmega_T - \{*\}} b^{\alpha_1} \cdot p^{\alpha_2} \,d\mu(b,p).
\]

Our starting point is the classical result that the moment generating function of a finite Borel measure on $\R^2$ with compact support completely characterizes the measure. In other words, computing all of the moments of a persistence measure would completely characterize the measure. However, as we have previously mentioned, certain moments are not Lipschitz-continuous with respect to the partial $1$-Wasserstein metric.

\begin{example}
    Let $\delta > 0$, and consider persistence daigrams with a single point $D_b = \{ (b, \delta)\}$, where $b \in [0, T]$. Note that the $(\alpha_1, 0)$-moment of $D_b$ is exactly $b^{\alpha_1}$. Let $b, b' \in [0,T]$ such that $b - b' > 2\delta$. Then, the partial $1$-Wasserstein distance between $D_b$ and $D_{b'}$ is $W_1^\partial[d](D_b, D_{b'}) = 2\delta$; however, the difference between their $(\alpha_1, 0)$-moments is $|b^{\alpha_1} - (b')^{\alpha_1}|$. Because $\delta$ can be made arbitrarily small, the ``pure-birth'' $(\alpha_1, 0)$-moments are not Lipschitz.
\end{example}

The crucial observation is that omitting such ``pure-birth moments'' still allows us to characterize measures, while being stable with respect to $p$-Wasserstein metrics; this observation was also made in earlier work on tropical coordinates on the space of persistence diagrams~\cite{kalisnik_tropical_2019, adcock_ring_2016}. Let $\R[[e_1, e_2]]$ denote the formal power series algebra in two (commutative) variables $e_1$ and $e_2$, where we denote an element $Q = \sum_{\alpha_1, \alpha_2=0}^\infty Q_{\alpha_1, \alpha_2} e_1^{\alpha_1} e_2^{\alpha_2} \in \R[[e_1, e_2]]$ by its coefficients $(Q_{\alpha})$. Similar to the tensor algebra, we equip this space with the (possibly infinite) inner product given by $\langle Q, R\rangle = \sum_{\alpha} Q_\alpha \cdot R_\alpha$, and restrict ourselves to the Hilbert space of finite norm elements
\begin{align}
    \cR \coloneqq \left\{Q \in \R[[e_1, e_2]] \, : \, \|Q\|< \infty\right\}.
\end{align}

\begin{definition}
    The \emph{polynomial map} $P: \wOmega_T \to \cR$ is defined by
    \begin{align*}
        P_{\alpha_1, \alpha_2}(b,p) = \left(\frac{1}{\alpha_1 ! \alpha_2!}\right)^{1/2} b^{\alpha_1} \cdot p^{\alpha_2} \quad \text{and} \quad P_{\alpha_1, \alpha_2}(*) = 0
    \end{align*}
    for $\alpha_1 \geq 0$ and $\alpha_2 \geq 1$, $P_{\alpha_1, 0} = 0$ for $\alpha_1 > 0$, and $P_{0,0} = 1$. 
    The \emph{persistence moments map} $\wP: \cM_{\fin,1}^+(\wOmega_T) \to \cR$ is defined by
    \begin{align*}
        \wP_{\alpha_1, \alpha_2}(\mu) \coloneqq \int_{\wOmega_T} P_{\alpha_1, \alpha_2}(b,p) \, d\mu(b,p).
    \end{align*}
\end{definition}

We note that this integral is always well defined since the measures are finite and the domain is bounded. Furthermore, the normalization factor in $P_{\alpha_1, \alpha_2}$ is used to ensure the image has finite norm, and to provide an efficient kernel computation.\medskip

\paragraph{Injectivity and Stability of Persistence Moments}

\begin{proposition}
    The persistence moments map $\wP: \cM_{\fin}^+(\wOmega_T) \to \cR$ is injective. 
\end{proposition}
\begin{proof}
    First, we note that the statement of the proposition is equivalent to saying that the polynomial map $P: \wOmega_T \to \cR$ is characteristic (\Cref{def:univ_char}) with respect to finite Radon measures. Let $\cG=C(\wOmega_T)$ be the space of continuous functions on $\wOmega_T$ equipped with the uniform topology. By Theorem~\Cref{thm:duality} and the fact that the dual $\cG$ contains finite Radon measures, we are reduced to proving that the polynomial map is universal.

    First, we note that the linear functionals $\langle \ell, P(\cdot)\rangle : \wOmega_T \to \R$, where $\ell \in \cR$, is an algebra (where the algebraic structure is inherited from the formal power series). Next, the collection of such polynomial functions vanishes nowhere on $\wOmega_T$ and separates points on $\wOmega_T$. Finally, the space $\wOmega_T$ is locally compact, and thus the locally compact Stone-Weierstrass theorem holds, and thus $P$ is universal. 
\end{proof}

\begin{proposition}
    The persistence moments map $\wP: \cM_{\fin}^+(\wOmega_T) \to \cR$ is Lipschitz.
\end{proposition}
\begin{proof}
    First, we note that the persistence moments map is the extension of the polynomial map to the Lipschitz-free space $\cF(\wOmega_T)$, restricted to the finite Radon measures $\cM_{\fin}^+(\wOmega_T)$ (which is contained in $\cF(\wOmega_T)$ by~\Cref{prop:measures_embedded_lfspace}). Thus, by the universal property of Lipschitz-free spaces~\Cref{thm:lipschitz_free}, it remains to prove that the polynomial map $P$ is Lipschitz. 

    We begin with the unnormalized monomials $U_{\alpha_1, \alpha_2}: \wOmega\to \R$ defined by $U_{\alpha_1, \alpha_2}(b,p) = b^{\alpha_1} \cdot p^{\alpha_2}$ where $\alpha_1 \geq 0$ and $\alpha_2 \geq 1$. Consider two points $z, z' \in \wOmega_T$. First, suppose that $\wid(z,z') = d(z,z')$. Then, we have
    \begin{align*}
        |U_{\alpha_1, \alpha_2}(z) - U_{\alpha_1, \alpha_2}(z')| \leq (\alpha_1 + \alpha_2) T^{\alpha_1+\alpha_2} d(z,z').
    \end{align*}
    Next, suppose $\tilde{d}(z,z') = d(z, \Delta_T) + d(z', \Delta_T)$. Let $x, x'$ be the projection of $z, z'$ to the x-axis $\Delta_T$. Then, we have
    \begin{align*}
        |U_{\alpha_1, \alpha_2}(z) - U_{\alpha_1, \alpha_2}(z')| &\leq |U_{\alpha_1, \alpha_2}(z) - U_{\alpha_1, \alpha_2}(*)| + |U_{\alpha_1, \alpha_2}(*) - U_{\alpha_1, \alpha_2}(z'))| \\
        & = (\alpha_1 + \alpha_2) T^{\alpha_1+\alpha_2}d(z, x) + (\alpha_1 + \alpha_2) T^{\alpha_1+\alpha_2}d(z',x') \\
        & = (\alpha_1 + \alpha_2) T^{\alpha_1+\alpha_2}\big( d(z, \Delta_T) + d(z', \Delta_T)\big).
    \end{align*}
    Therefore, $U_{\alpha_1, \alpha_2}$ is Lipschitz continuous with Lipschitz constant $(\alpha_1 + \alpha_2) T^{\alpha_1+\alpha_2}$. \medskip
    
    Now, considering the map $P$, we have
    \begin{align*}
        \|P(z) - P(z')\|^2 &\leq \sum_{m=1}^\infty \sum_{\substack{\alpha_1+\alpha_2 =m,\\ \alpha_2 \geq 1}} \frac{1}{\alpha_1! \alpha_2!} |U_{\alpha_1, \alpha_2}(z) - U_{\alpha_1, \alpha_2}(z')|^2 \\
        & \leq \sum_{m=1}^\infty \sum_{\substack{\alpha_1+\alpha_2 =m,\\ \alpha_2 \geq 1}} \frac{1}{m!} {m \choose \alpha_1} m^2 T^{2m} \tilde{d}(z,z')^2 \\
        & \leq \sum_{m=1}^\infty \frac{m^2 (2T^2)^{m} }{m!} \tilde{d}(z,z')^2,
    \end{align*}
    In order to bound the sum, we note that $m^2 \leq 3^m$ for all $m \geq 1$. Thus, we have 
    \begin{align*}
        C &= \left( \sum_{m=1}^\infty \frac{m^2 (2T^2)^{m} }{m!} \right)^{1/2} \leq \left(\sum_{m=1}^\infty \frac{ (6T^2)^{m} }{m!} \right)^{1/2} \leq \exp(3T^2).
    \end{align*}
    Therefore, the map $P$ is Lipschitz continuous.
\end{proof}

Combining the previous two results, we find that the restriction of the moment map to $\cM^+_{\fin}(\widetilde{\Omega}_T)$ is stable and injective for bounded persistence diagrams, which is sufficient for applications to data. While $\widetilde{P}$ is stable on the Lipschitz-free space $\cF(\widetilde{\Omega}_T)$ by the previous proposition, it is not immediate that the extension to the Lipschitz-free space is still injective. 

\begin{theorem}
\label{thm:moment_map_lipschitz}
    The persistence moments map $\wP : \cM^+_{fin}(\widetilde{\Omega}_T) \rightarrow \cR$ is Lipschitz and injective. 
\end{theorem}

\paragraph{Moment Truncation}
While all of the persistence moments are required for the map to be injective, it may be required to use a truncation of the moment map in practice,
\begin{align*}
    \widetilde{P}_{n}: \cM_{\fin}(\widetilde{\Omega}_T) \rightarrow \cR_n,
\end{align*}
where $\cR_n$ is the restriction of $\cR$ to the set of polynomials of degree at most $n$. Despite the loss of injectivity, the truncation error is small due to the factorial decay of the moments. 

\begin{lemma}
    Given $D = \{z_i = (b_i, p_i)\}_{i=1}^{N} \in \cD^+_\fin(\widetilde{\Omega})$, we have
    \begin{align*}
        \|\widetilde{P}(D) - \widetilde{P}_{n}(D)\| \leq \sqrt{\frac{N}{(n+1)!}} e^{T^2} \cdot (\sqrt{2} T)^{n+1}.
    \end{align*}
\end{lemma}
\begin{proof}
    Expanding the definition of the untruncated and truncated moment maps, we have
    \begin{align*}
        \|\widetilde{P}(D) - \widetilde{P}_{n}(D)\|^2 & \leq \sum_{m=n+1}^\infty \sum_{\substack{\alpha_1+\alpha_2 =m,\\ \alpha_2 \geq 1}} \frac{1}{\alpha_1! \alpha_2!}\sum_{i=1}^{N} \left|b_i^{\alpha_1} \cdot p_i^{\alpha_2}\right|^2 \\
        &\leq \sum_{i=1}^N \sum_{m=n+1}^\infty  \frac{(2T^2)^m}{m!}\\
        & \leq \frac{N \cdot e^{2T^2} (2T^2)^{n+1}}{(n+1)!}.
    \end{align*}
\end{proof}

\paragraph{Kernel Computation}
The previous result suggests that it is often enough to compute truncated moments; however, we can also compute the inner product between the untruncated moment map exactly.

\begin{proposition}
\label{prop:moment_kernel}
    Given two persistence diagrams $D , D' \in \cD^+_\fin(\widetilde{\Omega}_T)$, defined by
    \[
        D = \{z_i = (b_i, p_i)\}_{i=1}^N, \quad \quad D' = \{z_i = (b'_i, p'_i)\}_{i=1}^{N'},
    \]  
    the \emph{persistence moment kernel} can be computed as
    \begin{equation}
        \kappa(X, Y) \coloneqq \langle \wP(D), \wP(D')\rangle_{\cR} = \sum_{i=1}^{N_X} \sum_{j=1}^{N_Y} \exp{\langle z_i, z_j'\rangle} - \exp{b_{i}\cdot b'_{j}} + 1.
    \end{equation}
\end{proposition}
\begin{proof}
By definition of the moment map, and the fact that persistence diagrams are sums of Dirac measures, 
\begin{align*}
    \kappa(X,Y) & = \sum_{i=1}^{N_X} \sum_{j=1}^{N_Y} \langle P(z_i), P(z'_j) \rangle_\cR.
\end{align*}
Expanding the definition of the polynomial map, we find
\begin{align*}
    \langle P(z_i), P(z'_j) \rangle_\cR &= 1 + \sum_{m=1}^\infty \sum_{\substack{\alpha_1+\alpha_2 =m \\ \alpha \geq 0, \alpha \geq 1}} {m \choose \alpha} \frac{(b_i b'_j)^{\alpha_1} (p_i p'_j)^{\alpha_2}}{m!} \\
    & = \exp{\langle z_i, z_j'\rangle} - \exp{b_{i}\cdot b'_{j}} + 1
\end{align*}
\end{proof}

\paragraph{Persistence Moments with Signatures}
The continuity and injectivity of $\widetilde{P}_{\cM}$ allows us to treat static persistence diagrams as elements of the Hilbert space $\cR$. Then by~\Cref{thm:univ_char_general}, we have a universal and characteristic feature map for paths of persistence diagrams, viewed as elements of $\widetilde{C}^{1-var}(\cR)$.

\begin{corollary}[to~\Cref{thm:univ_char_general}]
    Let $\Lambda: T\ps{\cR} \rightarrow T\ps{\cR}$ be a tensor normalization. The normalized signature $\Lambda \circ S: \widetilde{C}^{1-var}(\cR) \rightarrow T\ps{\cR}$ is injective, universal with respect to bounded continuous functions $C_b(\widetilde{C}^{1-var}(\cR))$ with respect to the strict topology, and characteristic with respect to the space of finite Borel measures on $\widetilde{C}^{1-var}(\cR)$.
\end{corollary}

Furthermore, the composition of the truncated path signature with the moment map is Lipschitz continuous for paths with a fixed bounded length.

\begin{theorem}
    Suppose $L > 0$ and suppose
    \begin{align*}
        C^{1-var}_L(V) \coloneqq \{ \gamma \in C^{1-var}(V) \, : \, |\gamma|_{1-var} \leq L\}.
    \end{align*}
    Then, the composition $S_{\leq M} \circ \wP : C^{1-var}_L(\cM^+_{\fin}(\widetilde{\Omega}_T)) \rightarrow T\ps{\cR}$ is Lipschitz continuous.
\end{theorem}
\begin{proof}
    By~\Cref{thm:moment_map_lipschitz}, $\widetilde{P}_{\cM}$ is Lipschitz continuous, and let $C$ be its Lipschitz constant. Then, given two paths $\gamma_1, \gamma_2 \in C^{1-var}(\cM_{fin}(\widetilde{\Omega}_T))$, we have
    \begin{align*}
        |\widetilde{P}(\gamma_1 - \gamma_2)|_{1-var} & = \sup_{\Pi}\sum_{i=1}^n \|\widetilde{P}(\gamma_1(t_i) - \gamma_2(t_i)) - \widetilde{P}(\gamma_1(t_{i-1}) - \gamma_2(t_{i-1}))\| \\
        & \leq \sup_{\Pi}\sum_{i=1}^n C \|(\gamma_1(t_i) - \gamma_2(t_i)) - (\gamma_1(t_{i-1}) - \gamma_2(t_{i-1}))\| \\
        & \leq C |\gamma_1 - \gamma_2|_{1-var}.
    \end{align*}
    Then, combining this with the stability of signatures in~\Cref{thm:signature_stability}, we have
    \begin{align*}
        \|S_{\leq M} \circ \widetilde{P}(\gamma_1) - S_{\leq M} \circ \widetilde{P}(\gamma_2)\| & \leq \frac{C \beta}{L} \exp\left(\left(\frac{L}{\sqrt{2}\beta}\right)^2\right) |\gamma_1 - \gamma_2|_{1-var}.
    \end{align*}
\end{proof}

\subsection{The Discrete Path Signature Kernel}
While the path signature is naturally defined for continuous paths, we often work with discrete time series in data science applications. There is a discrete approximation of the continuous signature for finite dimensional discrete paths which allows for more efficient computation, originally introduced in~\cite{kiraly_kernels_2019}. In this subsection, we recall main concepts of this approximation. \medskip

In this subsection, we let $V$ denote a Hilbert space with orthonormal basis $\{e_\alpha\}_{\alpha \in A}$ for some index set $A$. For $L \in \N$, we let $[L] \coloneqq \{1, \ldots, L\}$. A \emph{discrete path} is a map $\hat{\gamma}:[L+1] \rightarrow V$, and its \emph{discrete derivative} $\hat{\gamma}': [L] \rightarrow V$ is defined to be
\begin{align*}
    \hat{\gamma}'(t) = \hat\gamma(t+1) - \hat\gamma(t)
\end{align*}
for all $t \in [L]$. We let $\Seq(V)$ denote the space of discrete paths of arbitrary finite length on $V$. Furthermore, for $\alpha \in A$, we denote $\gamma_\alpha: [L+1] \rightarrow \R$ to be the $\alpha$-component of the path.

\begin{definition}
\label{def:discrete_ps}
    Let $\hat{\gamma}:[L+1] \rightarrow V$. Let $I = (i_1, \ldots, i_m) \in A^m$ be a multi-index. We define the discrete $m$-simplex with length $L$ to be
    \begin{align*}
        \hat{\Delta}^m_L = \{(t_1, \ldots, t_m) \in [L]^m \quad : \quad 0 \leq t_1 < t_2 < \ldots < t_m \leq L\}.
    \end{align*}
    The discrete path signature of $\hat{\gamma}$ with respect to $I$ is defined to be
    \begin{align*}
        \hat{S}^I(\hat{\gamma}) \coloneqq \sum_{(t_1, \ldots, t_m) \in \hat{\Delta}^m_L} \hat{\gamma}_{i_1}'(t_1)\ldots \hat{\gamma}_{i_m}'(t_m).
    \end{align*}
\end{definition}

The discrete signature is a map $\hat{S}: \Seq(V) \rightarrow T\ps{V}$, or in truncated form as $\hat{S}_{\leq M}: \Seq(V) \rightarrow T^{(M)}(V)$. We will be using the truncated variant of the signature in data analysis applications due to its computability.  The truncated signature feature map then determines a kernel $\hat{K}_{\leq M} : \Seq(V) \times \Seq(V) \rightarrow T^{(M)}(V)$ defined for $\hat{\alpha}, \hat{\beta} \in \Seq(V)$ and $M \geq 1$ by
\begin{align*}
    \hat{K}_{\leq M}(\hat{\alpha}, \hat{\beta}) = \langle \hat{S}_{\leq M}(\hat{\alpha}), \hat{S}_{\leq M}(\beta)\rangle.
\end{align*}
Efficient algorithms to compute the truncated discrete signature kernel along with discretization error results relating the continuous and discrete signatures are provided in~\cite{kiraly_kernels_2019,chevyrev_signature_2022}.

\subsection{Signature Kernels for Paths of Persistence Diagrams}
\label{ssec:signaturekernels_pdpaths}
Combining the previous results in this section, we obtain a feature map for discrete paths of persistence diagrams into the Hilbert space $T^{(M)}(\cR)$
\begin{align*}
    \hat{S}_{\leq M} \circ \widetilde{P} :\Seq(\cM^+_{\fin}(\widetilde{\Omega}_T)) \rightarrow T^{(M)}(\cR). 
\end{align*}
This defines a kernel for sequences of bounded persistence diagrams, which can be computed by using the truncated form of the moment map $\widetilde{P}$ to obtain explicit features in $\cR_n$. Then, we use Algorithm 3 in~\cite{kiraly_kernels_2019} to compute the truncated discrete signature kernel. 

While we have primarily discussed the moment map as the feature map for persistence diagrams, the signature method can be used together with any feature map such as persistence landscapes~\cite{bubenik_statistical_2015-1}, persistence images~\cite{adams_persistence_2017}, or tropical coordinates~\cite{kalisnik_tropical_2019}. Although we do not pursue it here, we can also lift kernels for static persistence diagrams, such as the sliced-wasserstein kernel~\cite{carriere_sliced_2017} and the persistence scale space kernel~\cite{reininghaus_stable_2015, kwitt_statistical_2015}, to kernels for dynamic persistence diagrams, as discussed in~\cite{kiraly_kernels_2019,chevyrev_signature_2022}. \medskip

\section{Application: Parameter Estimation in a Swarm Model}
\label{sec:applications}
In this section, we study the paths of persistence diagrams which arise in models of collective motion. We consider a swarm model and use kernel support vector regression to estimate model parameters purely by using the evolution of the \emph{shape} of the swarm over time. In particular, all of our methods only use unidentified position data from the swarms: we do not have information regarding how individual particles move over time. This restriction disqualifies the use of classical order parameters which make use of velocity information to classify collective behavior.

We study the 3D D'Orsogna model~\cite{dorsogna_self-propelled_2006} for collective motion. \Cref{fig:swarm} depicts a single swarm simulation at various time steps, and demonstrates how the global shape of the swarm evolves over time. The exact organization of the swarm over time will depend on the $(C, \ell)$ parameters, which we have chosen to be $(0.6, 0.3)$ in the figure. The stable structures of these swarms are studied and organized into a phase diagram in~\cite{nguyen_thermal_2012}.

\begin{figure}
\centering
		\includegraphics[width=0.8\textwidth]{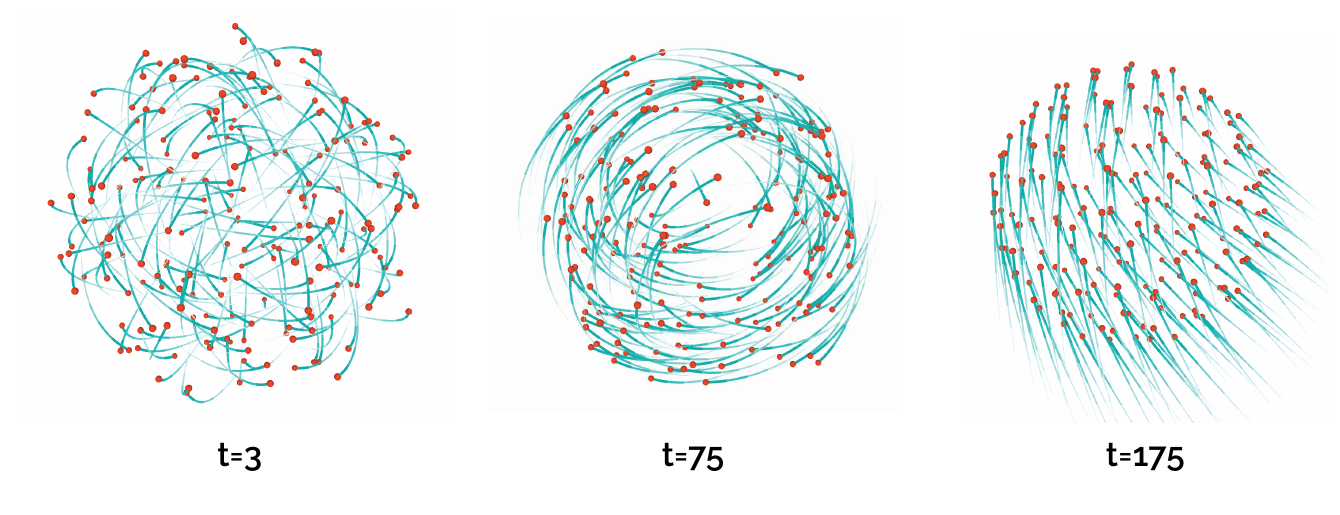}
    		\caption{A single swarm simulation shown at time points $t=3$, $t=75$, and $t=175$. This is simulated using the 3D D'Orsogna model~\cite{dorsogna_self-propelled_2006} with $\Nagent = 200$ agents. The model parameters are $m=1$, $\alpha = 1.0$, $\beta = 0.5$, $C=0.6$, and $\ell = 0.3$.}
    		\label{fig:swarm}
\end{figure}
Previous studies~\cite{bhaskar_analyzing_2019, xian_capturing_2020} have successfully applied persistent homology to classify swarms based on the 2D D'Orsogna and 2D Vicsek models, using a feature map called CROCKER plots. Both of these studies use a collection of simulations generated using a fixed finite set of model parameters, and use clustering methods to classify the simulations into clusters based on the fixed set of parameters. In contrast, our experiments use the 3D D'Orsogna model, and our model parameters are chosen uniformly at random within a specified region in parameter space. Rather than using classification algorithms which only work for a finite number of model parameters, we use regression to perform parameter estimation.

We study the generalizability and stability of our methods by considering situations with missing data in two dimensions. First, we consider agent-wise subsampling, in which a fixed number of agents are randomly subsampled at each time step. Second, we consider time subsampling, in which a fixed number of time steps are subsampled for each simulation trial. Furthermore, we perform experiments with heterogeneous train/test splits, in which the training is performed on the full simulation and the test set consists of simulations with missing data. We find that the path signature (combined with some vectorization of static persistence diagrams) outperforms the CROCKER plots in a large majority of experiments. Furthermore, we also find that using persistence moments is competitive with other standard vectorizations, despite being a much lower dimensional vectorization.

\subsection{Model Description and Data}
The D'Orsogna model~\cite{dorsogna_self-propelled_2006} describes the collective motion of $\Nagent$ interacting agents with positions $\bx_i \in \R^d$, and velocities $\bv_i \in \R^d$. Previous work~\cite{topaz_topological_2015, bhaskar_analyzing_2019} has applied persistent homology techniques to study the two-dimensional ($d=2$) model~\cite{bhaskar_analyzing_2019}. Here we choose to focus on the three-dimensional ($d=3$) case. The agents are governed by the following coupled differential equations
\begin{align*}
    \frac{d\bx_i}{dt} & = \bv_i, \quad \quad
    m\frac{d\bv_i}{dt} = \left(\alpha - \beta|\bv_i|^2\right)\bv_i - \nabla_i U(\bx_i), \\
    U(\bx_i) & = \sum_{j \neq i}^N \left(C_r e^{-\br_{i,j}/\ell_r} - C_a e^{-\br_{i,j}/\ell_a}\right),
\end{align*}
where $\br_{i,j} = \|\bx_i - \bx_j\|$ and $\nabla_i$ refers to the gradient with respect to $\bx_i$. Each agent has mass $m$, is self-propelled with propulsion strength $\alpha$, and experiences a velocity-dependent drag with coefficient $\beta$. Furthermore, each particle $\bx_i$ experiences a potential $U(\bx_i)$ derived from pairwise interactions with all of the other agents. Each pair has a \emph{repulsive} component with strength $C_r >0$ and characteristic length scale $\ell_r>0$, and an \emph{attractive} component with strength $C_a > 0$ and characteristic length scale $\ell_a > 0$. \medskip

For our simulations, we use $\Nagent=200$ interacting agents. Furthermore, we fix $m=1$, $\alpha = 1.0$, and $\beta = 0.5$ which corresponds to the parameters chosen in~\cite{nguyen_thermal_2012}. After nondimensionalizing the equation, we are left with two free parameters: the interaction strength $C = C_r / C_a$ and the characteristic length $\ell = \ell_r / \ell_a$.

We perform $\Nsimulation =500$ instances of this simulation, where each instance is run for $t \in [0, 400]$, and is discretized using $T=200$ uniformly spaced time points. The ODEs are solved using the \texttt{DifferentialEquations} package in Julia. The initial positions are uniformly sampled from the unit cube $[0,1]^3$, each component of the initial velocities is sampled from a Gaussian with mean $1$ and variance $1$, following~\cite{nguyen_thermal_2012}, and $C$ and $\ell$ are chosen uniformly at random from $[0.1, 2]$. We reject unbounded phenotypes, defined to be a simulation in which any particle moves 40 units in any direction before $t = 200$, since they can be easily identified with a simple scale statistic. This restricts the distribution of $(C, \ell)$, and the distribution used in our simulations is shown in~\Cref{fig:CL}. We consider missing data in two different ways.

\begin{figure}
\centering
		\includegraphics[width=0.25\textwidth]{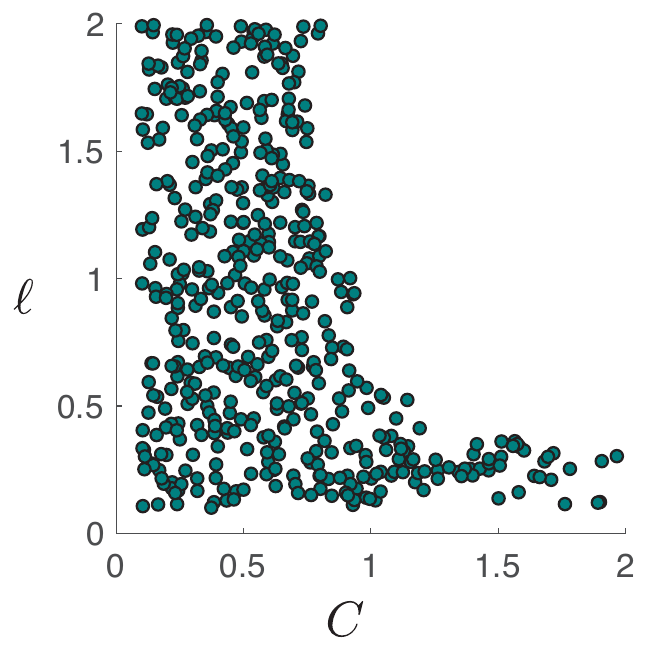}
    		\caption{Scatter plot of generated parameter pairs $(C, \ell)$ that resulted in bounded swarm simulations.}
    		\label{fig:CL}
\end{figure}

\begin{enumerate}
    \item \textbf{(Agent Subsample)} Subsample $N=50$ agents uniformly at random at each timestep
    \item \textbf{(Temporal Subsample)} Subsample $T=50$ or $20$ time points for each simulation. The temporal subsampling is performed in two ways:
    \begin{enumerate}
        \item \textbf{(Random Temporal Subsample)} The time points are subsampled uniformly at random.
        \item \textbf{(Initial Temporal Subsample)} The first $T$ time point are subsampled. This is used to model the case where only the initial dynamics of the swarm are observed, limiting reliance on steady-state dynamics for parameter estimation.
    \end{enumerate}
\end{enumerate}

\subsection{Feature Maps}

We compute persistent homology with the Vietoris-Rips filtration up to dimension $d=2$ for all simulations and all time steps in order to obtain three paths of persistence diagrams (at dimensions $d=0, 1, 2$) for each simulation. Persistent homology is computed using the \texttt{Eirene} software package~\cite{henselman_matroid_2017} written in Julia. We compute the following features for each path of persistence diagrams; path signatures are truncated at level $m=3$. Landscapes and images are computed using the \texttt{giotto-tda} package~\cite{tauzin_giotto-tda_2021}. We refer the reader to the  references for further details on the other vectorizations.

\begin{enumerate}
    \item \textbf{Persistence Moments + Signature.} Moments are computed for each homology dimension\footnote{For dimension $0$, we only compute pure lifetime moments, since all birth times are $0$.}, truncated at degree $2$, and then concatenated into a single vector. In particular, there are 2 features for $H_0$, 3 features for $H_1$, and 3 features for $H_2$
    \item \textbf{Persistence Paths + Signature.} Persistence paths~\cite{chevyrev_persistence_2020} views static persistence diagrams of various dimensions as a \emph{Betti curve} $\beta = (\beta_0, \beta_1, \beta_2)$: the Betti numbers of the persistence module over the scale parameter. The signature of this Betti curve is computed and used as feature for a static diagram. Here, there are 3 features at level $1$ and 9 features at level $2$.
    \item \textbf{Persistence Landscapes + Signature.} Persistence landscapes~\cite{bubenik_statistical_2015-1} view persistence diagrams as a sequence of curves which quantify the containment of points on the persistence diagrams within certain intervals of the scale parameter. We use $5$ landscapes per dimension, discretized using $10$ time points on a log-scale due to large differences in scale between simulations. 
    \item \textbf{Persistence Images + Signature.} Persistence images~\cite{adams_persistence_2017} view persistence diagrams as a function $PI: \Omega_T \to \R$ defined as a mixture of Gaussians centered at each point of the persistence diagram. The region $\Omega_T$ is then discretized to obtain a vectorized feature. We use a $10\times 10$ grid on a log-scale due to large differences in scale between simulations, and Gaussians with variance $\sigma=0.3$. 
    \item \textbf{CROCKER Plots.} Crocker plots, introduced in~\cite{topaz_topological_2015}, treat the discretized Betti curve $\beta = (\beta_0, \beta_1, \beta_2)$ as a feature vector. Following~\cite{bhaskar_analyzing_2019}, we discretize the Betti curve on a log scale from $\epsilon = 10^{-4}$ to $\epsilon = 1$ with $N_\epsilon = 100$ points. Furthermore, we uniformly subsample the time coordinate to obtain $N_t = 20$ points. Collapsing the discretized Betti curves into a single vector, each path of persistence diagrams is an element in $\R^{3\times 100 \times 20} = \R^{6000}$. The kernel is the standard inner product on $\R^{6000}$. 
\end{enumerate}

For each of the static vectorization methods, we chose parameters such that the resulting representation was low-dimensional, yet was still able to capture relevant information about the data. A key advantage of persistence moments and paths is the lack of discretization, resulting in significantly fewer features, as seen in~\Cref{tab:dimension_features}.

\begin{table}[h]
    \centering
\begin{tabular}{ccccc} \toprule
    & Moments & Paths & Landscapes & Images \\ \midrule
    Dimension of Features & 8 & 12 & 150 & 300 \\ \bottomrule
\end{tabular}
\caption{Dimensions of the static vectorization features, including persistence diagrams of dimension $0, 1$, and $2$.}
\label{tab:dimension_features}
\end{table}

\subsection{Regression and Statistical Methods}
We use the support vector regression (SVR) algorithm~\cite{drucker_support_1997, scholkopf_learning_2002} to take advantage of signature kernels. We use the SVR algorithm from the Julia implementation of the \texttt{scikit-learn} library. 
We begin by enumerating the $\Nsimulation = 500$ individual simulations and randomly select a 400 / 100 split for the training and test data respectively. With this choice, we perform a \emph{trial} of our analysis, which consists of three steps:
\begin{enumerate}
    \item \textbf{Hyperparameter selection.} The SVR algorithm uses the hyperparameters $(\lambda, \epsilon)$ and we tune select these by grid search with 4-fold cross-validation exclusively using the training data. The grid points of $\lambda$ are logarithmically spaced between $10^{-3}$ and $10^1$, and those of $\epsilon$ are logarithmically spaced between $10^{-4}$ and $10^0$. We select 5 points for each parameter for a total of $25$ parameter pairs.
    
    For the signature methods, we use a sliding window embedding, a common preprocessing step for signatures~\cite{chevyrev_primer_2016, fermanian_embedding_2021}. In particular, also optimize over $k = 0, 1, 2$ lags. 
    \item \textbf{Training.} We select the parameter pair $(\lambda, \epsilon, k)$ with lowest mean square error (MSE). We then train the SVR using these hyperparameters on the entire training set.
    \item \textbf{Prediction.} Using the trained regressor, we predict the $(C, \ell)$ values of the data in the test set and compute the MSE for each parameter individually.
\end{enumerate}

\begin{figure}[t]
    \centering
            \includegraphics[width=0.9\textwidth]{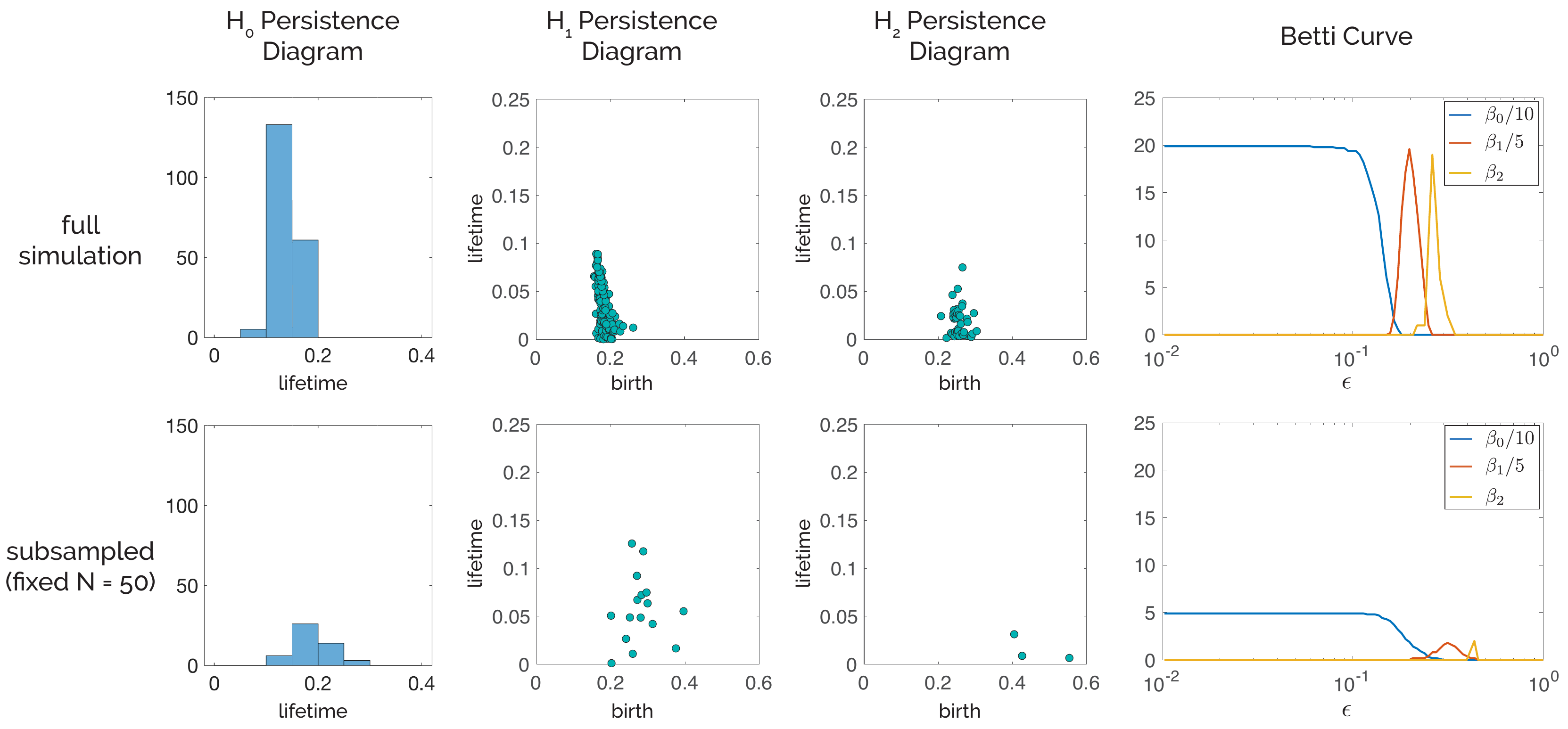}
                \caption{Comparison of persistence diagrams and Betti curves between a point cloud from the full simulation (top row) and the subsampled simulation with fixed $N=50$ (bottom row). The $H_0$ persistence diagram is shown as a histogram of lifetimes, since all birth times are at $\epsilon = 0$. In the Betti curve, $\beta_0$ is scaled by a factor of $1/10$ and $\beta_1$ is scaled by a factor of $1/5$ for visualization purposes.}
                \label{fig:subsample_comparison}
\end{figure}

\subsection{Heterogeneous Train/Test Experiments}

In order to further study the generalizability of our method, we consider experiments with heterogeneous train/test data. In these experiments, the training set consists of full simulations without any subsampling, while the test set consists of agent-wise and/or temporally subsampled data. For such experiments, a direct application of our methods would not perform well due to known scaling properties of persistence diagrams. A representative example of this subsampling behavior is shown in~\Cref{fig:subsample_comparison}, demonstrating two representative phenomena that occur due to subsampling:
\begin{enumerate}
    \item the persistence diagrams are much sparser due to fewer points in the point cloud, and
    \item there is a shift in the birth coordinate.
\end{enumerate}

We can understand these two phenomena through limit theorems for point processes in a unit $d$-cube, as shown in~\cite{divol_choice_2019,divol_understanding_2021}. In particular, given a probability measure $\PP$ supported on the unit $d$-cube, and point clouds $\X_n = (X_1, \ldots, X_n)$ with $X_n \in [0,1]^d$ independently sampled from this measure. Further, let $\mu_n = \frac{1}{n}\mathrm{Dgm}(n^{1/d}\X_n) \in \cM^+(\widetilde{\Omega})$ be the scaled persistence measure of $\X_n$. Then~\cite{divol_understanding_2021} shows that there exists a measure $\mu$ such that $\mu_n \rightarrow \mu$ in the $p$-partial Wasserstein metric.

While our point clouds are not subsampled from a unit $d$-cube, this limit theorem motivates a feature normalization to amend the subsampling phenomena. In particular, using the fact that our swarm models evolve in $\R^3$, we scale our point clouds by $N^{1/3}$ and the resulting persistence diagrams by $N^{-1}$, where $N$ is the number of subsampled agents. All heterogeneous experiments use persistence diagrams which are normalized in this way.

\begin{figure}[t]
    \centering
            \includegraphics[width=0.98\textwidth]{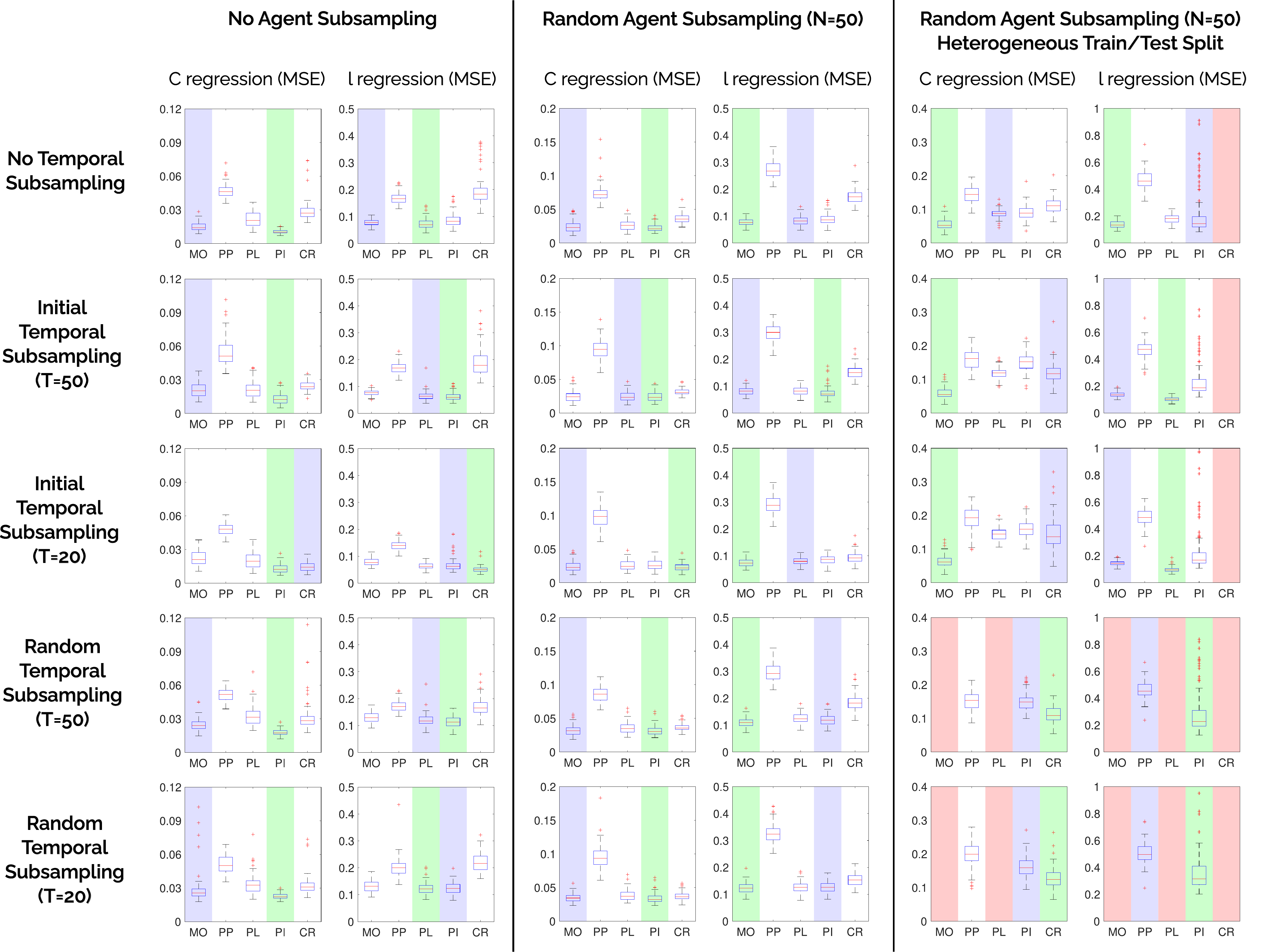}
                \caption{Box plots of the MSE statistics of the $C$ and $\ell$ regression, computed using $\Ntrial = 100$ trials for each experiment. The feature map with the lowest (resp. second lowest) median MSE is highlighted in green (resp. blue); experiments where the MSE is greater than $0.4$ ($C$ regression) or $0.8$ ($\ell$ regression) are omitted and highlighted in red. Note that the $y$-axis is only fixed for each column. MO: persistence moments, PP: persistence paths, PL: persistence landscapes, PI: persistence images, CR: CROCKER plots.}
                \label{fig:stats}
\end{figure}
\subsection{Results and Discussion}

The results of our experiment are shown as boxplots in~\Cref{fig:stats}. We highlight two main observations.

\begin{enumerate}
    \item \textbf{Signature methods outperform CROCKER plots.} In a large majority of the experiments, most of the signature methods outperform the CROCKER plots, demonstrating the efficacy of the general path signature methodology. Persistence images paired with the signature consistently perform well throughout all experiments.
    \item \textbf{Persistence moments are competitive with classical static vectorizations.} In many of the experiments, persistence moments achieve second-best performance, with only a small difference in error compared to the best. Note that the dimension of persistence moments is an order of magnitude lower than both landscapes and images. 
    \item \textbf{Persistence moments perform well on heterogeneous experiments with different time intervals.} In the heterogeneous experiments where train and test split contain different time intervals (first three plots in third column in~\Cref{fig:stats}), moments consistently perform well, and outperform other methods in the $C$ regression. 
\end{enumerate}

We find that the combination of the moment map with the path signature is able to effectively capture the dynamic topological structure of these swarm models in order to accurately estimate the model parameters. An interesting direction for future work is to understand the critical properties of feature maps for persistence diagrams which lead to more robust path features when combined with the path signature. 

\begin{figure}[t]
    \centering
            \includegraphics[width=0.62\textwidth]{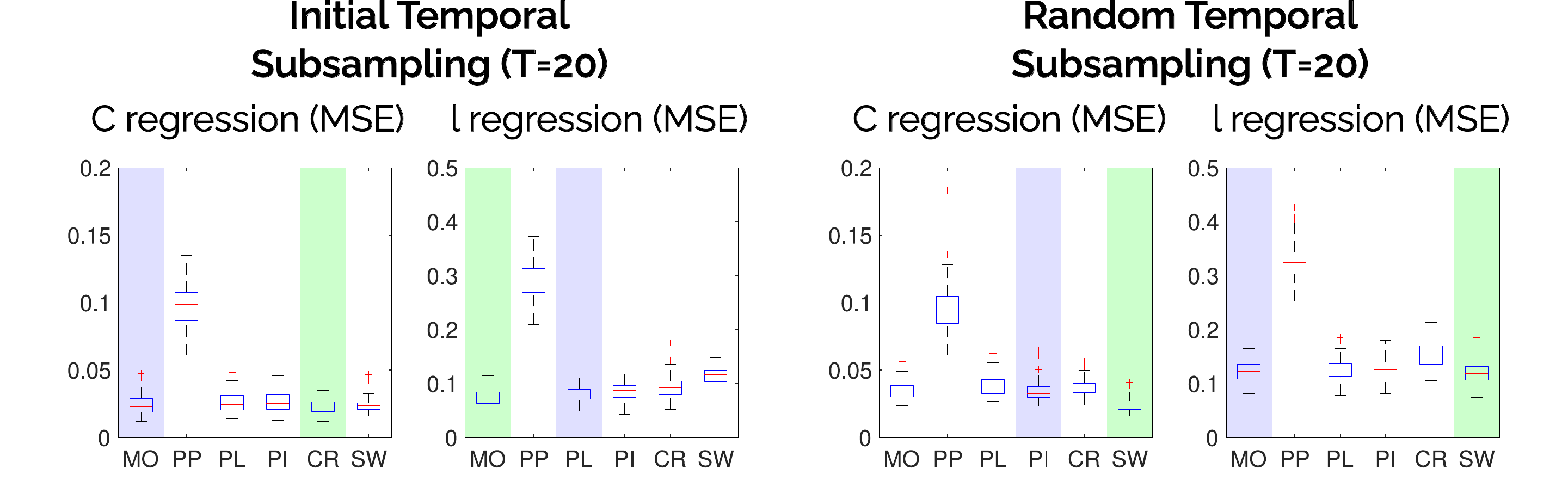}
                \caption{Box plots of the MSE statistics of the $C$ and $\ell$ regression, computed using $\Ntrial = 100$ trials for each experiment. Both experiments have random agent subsampling of $N=50$. The feature map with the lowest (resp. second lowest) median MSE is highlighted in green (resp. blue). MO: persistence moments, PP: persistence paths, PL: persistence landscapes, PI: persistence images, CR: CROCKER plots, SW: sliced Wasserstein kernel (for measures).}
                \label{fig:sw_stats}
\end{figure}

\subsection{Comparison with Kernel Mean Embedding Methods} \label{ssec:kme}
We also compare our dynamic methods with kernel mean embedding (KME) methods. In particular, suppose $D :[L] \to \cD^+_{\fin}$ is a discrete path of persistence diagrams. We forget all of the temporal information in the path by viewing this as a probability measure $\mu_D \in \cP(\cD_{\fin}^+)$ defined by $\mu_D \coloneqq \frac{1}{L} \sum_{t=1}^L \delta_{D(t)}$.
We consider the kernel mean embedding of this measure with respect to the sliced Wasserstein (SW) kernel~\cite{carriere_sliced_2017}. In particular, if $\Phi_{SW}: \cD_{\fin}^+ \rightarrow \cH_{SW}$ is the feature map corresponding to the SW kernel, the mean embedding of $\mu_D$ is
\[
    \E[\Phi_{SW}(\mu_D)] = \frac{1}{L} \sum_{t=1}^L \Phi_{SW}(D(t)).
\]
We use the formalism of \emph{support measure machines}~\cite{muandet_learning_2012}, and define a kernel on probability measures $\mu_D, \mu_{D'} \in \cP(\cD_{\fin}^+)$ by
\[
    \kappa_{\cP}(\mu_D, \mu_D') = \langle \E[\Phi_{SW}(\mu_D)], \E[\Phi_{SW}(\mu_{D'})]\rangle.
\]
For two paths of diagrams $D:[L] \to \cD^+_{\fin}$ and $D': [L'] \to \cD^+_{\fin}$, this is computed exactly as
\[
    \kappa_{\cP}(\mu_D, \mu_{D'}) = \frac{1}{L\cdot L'} \sum_{i=1}^L \sum_{j=1}^{L'} \kappa_{SW}(D(i), D'(j)),
\]
where $\kappa_{SW}$ is the SW kernel. We can then use this kernel in our above experiments. However, due to a longer computation time, we only perform the experiments with agent subsampling of $N=50$ and temporal subsampling of $T=20$; see~\Cref{apx:kme_time} for further details on time complexity.

Surprisingly, these KME methods are competitive with dynamic methods on both experiments. Previous articles which have studied collective motion using persistent homology have compared dynamic topological methods with classical order parameters~\cite{bhaskar_analyzing_2019,xian_capturing_2020}; however, to the authors' knowledge, experiments using KME methods for persistence diagrams derived from swarm data have not been previously reported. The initial experiments in this subsection suggest that further study is warranted, though this is beyond the scope of this article.

\section{Conclusion}
\label{sec:conclusion}

We have presented a theoretical and computational framework for the study of paths of persistence diagrams. By building a connection to Lipschitz-free spaces, we integrate several generalizations of the space of persistence diagrams into a single Banach space. This allowed us to define a Banach space of bounded variation paths of persistence diagrams, and thus to define their path signature. The Banach space structure provides a setting to study probabilistic and statistical properties of both static and time-varying persistence diagrams. Furthermore, the path signature is universal and characteristic, providing a robust tool kit for characterizing probability measures on paths of persistence diagrams. 

In addition, we introduced a new feature map for static persistence diagrams, motivated by the perspective of diagrams as measures, and showed that it is both injective and stable. Coupling this with the discrete path signature, we defined an approximation to the intrinsic continuous signature. This approximation is computable, stable, and valued in a Hilbert space, enabling the use of kernel methods for data analysis. We demonstrated the efficacy of our feature map with a parameter estimation problem for a 3D model of collective motion. 

While we specifically considered the moment map, the signature methods we presented is general: the signature can be used in conjunction with any feature map or kernel for persistence diagrams. Moreover, the framework laid out in this paper suggests several possibilities for future work, and we highlight some promising directions here. 
\begin{itemize}
    \item \textbf{Persistence vineyards.} Persistence vineyards~\cite{cohen-steiner_vines_2006} are paths of persistence diagrams with finer information; in particular it retains information about how individual homology classes evolve over time. Can we
    use methods such as adapted Wasserstein metrics~\cite{backhoff-veraguas_all_2020} and higher rank signatures~\cite{bonnier_adapted_2020} to characterize vineyards?

    \item \textbf{Rough paths.} The theory of \emph{rough paths}~\cite{lyons_differential_2007, lyons_system_2007} provides a deterministic integration theory for highly irregular paths valued in Banach spaces, and has led to several developments in stochastic differential equations~\cite{friz_multidimensional_2010, friz_course_2020}. Therefore, using the Lipschitz-free space as a Banach space for persistence diagrams, we may define and study \emph{rough paths of persistence diagrams}. This opens up the possibility of studying the persistent homology of data driven by stochastic differential equations.

    \item \textbf{Lipschitz-free spaces.} Although we have discussed several properties of Lipschitz-free spaces which are directly applicable to the study of persistence diagrams, are there other aspects of this connection which we can exploit? One concrete problem is the an explicit description of the Lipschitz-free space of the quotient metric space of a metric pair $(X, d, A)$. Can we adapt the representations for known Lipschitz-free spaces such as $\cF(\R^N)$~\cite{cuth_isometric_2017} to develop a new representation of persistence diagrams?
    
\end{itemize}

\appendix

\section{Notation}

\begin{table}[h]
    \centering
\begin{tabular}{ll} \toprule
    \multicolumn{2}{c}{Spaces of Persistence Diagrams}\\ \midrule
    $(Z,d,A)$ & arbitrary metric pair \\
    $(\Omega,d, \Delta)$ & metric pair for unbounded persistence diagram (birth, death)\\
    $(\Omega_T, d, \Delta_T)$ & metric pair for bounded persistence diagrams (birth, persistence)\\
    $(\widetilde{Z},\wid, *)$ & quotient (pointed) metric space of $(Z,d,A)$ \\
    $\cD^+_{\fin}(Z,A)$ & finite persistence diagrams on $(Z,A)$ \\
    $\cM^+_{\fin}(Z,A)$ & finite persistence measures (Radon measures) on $(Z,A)$ \\
    $\cM^+_{\fin, p}(Z,A)$ & finite $p$-persistence measures on $(Z,A)$ \\
    $\cF(\widetilde{Z})$ & Lipschitz-free space of $\widetilde{Z}$\\
    $W_p^{\partial}[d]$ & partial $p$-Wasserstein distance with respect to $d$\\
    $\cR$ & formal (commutative) power series of two variables with finite norm\\
    $P$ & polynomial map \\
    $\wP$ & persistence moments map \\  \midrule

    \multicolumn{2}{c}{Paths and Signatures}\\ \midrule
    $C^{1-var}(V)$ & bounded variation paths in $V$ \\
    $T\ps{V}$ & tensor algebra (power series of tensors with finite norm)\\
    $T^{(M)}(V)$ & tensor algebra truncated at degree $M$ \\
    $S$ & path signature \\
    $S_{\leq M}$ & path signature truncated at degree $M$ \\
    $\Lambda$ & tensor normalization\\
    $\hat{S}$ & discrete path signature\\
    $\hat{K}$ & discrete signature kernel\\ \bottomrule
\end{tabular}
\caption{Table of notation.}
\label{tab:notation}
\end{table}

\section{The Strict Topology}
\label{apxsec:strict}
The main tool to prove universality of feature maps is the Stone-Weierstrass theorem, which usually only holds for locally compact spaces. However, spaces such as path spaces are usually not locally compact. As was first noted in~\cite{chevyrev_signature_2022} for applications to signatures, when $\cG = C_b(\cX)$ is equipped with the \emph{strict topology}~\cite{giles_generalization_1971}, there exists a Stone-Weierstrass result, and $\cG'$ includes the space of probability measures on $\cX$.\medskip

Let $\cX$ be a topological space. A function $\psi: \cX \rightarrow \R$ \emph{vanishes at infinity} if for all $\epsilon >0$, there exists a compact set $K \subset \cX$ such that $\sup_{x \in \cX/K} |\psi(x)| < \epsilon$. Let $B_0(\cX,\R)$ denote the set of functions that vanishes at infinity. The stricrt topology on the space of continuous bounded functions $C_b(\cX,\R)$ is the topology generated by the seminorms
\[
    p_\psi(f) = \sup_{x \in \cX}|f(x)\psi(x)|, \psi \in B_0(\cX,\R).
\]

\begin{theorem}[\cite{giles_generalization_1971}]
\label{thm:strict_topology}
Let $\cX$ be a metrizable topological space.
\begin{enumerate}
    \item The strict topology on $C_b(\cX)$ is weaker than the uniform topology and stronger than the topology of uniform convergence on compact sets.
    \item If $\cG_0$ is a subalgebra of $C_b(\cX)$ such that for all $x, y\in \cX$, there exists some $f \in \cG_0$ such that $f(x) \neq f(y)$ ($\cG_0$ separates points), and for all $x \in \cX$, there exists some $f \in \cG_0$ such that $f(x) \neq 0$, then $\cG_0$ is dense in $C_b(\cX)$ under the strict topology.
    \item The topological dual of $C_b(\cX)$ equipped with the strict topology is the space of finite regular Borel measures on $\cX$. 
\end{enumerate}
\end{theorem}

\section{Quotient Metric Spaces and Measures}
\label{apx:quotient_metric}

We show that the quotient map for metric spaces induces an isometry for Radon measures. 

\begin{definition}
    Let $(X,d)$ be a metric space and suppose $\mu, \nu \in M^+(X)$ such that $\mu(X) = \nu(X)$. Let $\pi_1, \pi_2 : X \times X \rightarrow X$ be the projections to the first and second factor. The set of \emph{couplings} between $\mu$ and $\nu$ is defined to be
    \begin{align*}
        \Cpl(\mu, \nu) \coloneqq \left\{ \sigma \in \cM^+_{\fin}(X \times X) \, : \, (\pi_1)_*\sigma = \mu, (\pi_2)_*\sigma = \nu \right\}.
    \end{align*}
    The \emph{$p$-Wasserstein distance} between $\mu$ and $\nu$ is
    \begin{align*}
        W_p[d](\mu, \nu) \coloneqq \inf_{\sigma \in \Cpl(\mu, \nu)} \left( \int_{X \times X} d(x,y)^p d\sigma(x,y) \right)^{1/p}.
    \end{align*}
\end{definition}
The following proposition in~\cite{divol_understanding_2021} shows that we can compute partial Wasserstein distances using the ordinary Wasserstein distance for measures. 

\begin{proposition}[\cite{divol_understanding_2021}]
\label{prop:partial_ordinary_wass}
    Let $(X, d, A)$ be a metric pair and let $(\widetilde{X}, \tilde{d}, *)$ be its quotient. Let $\mu, \nu \in \cM^+_{fin, p}(X)$ and $r_\mu = \mu(X-A), r_\nu =  \nu(X-A)$, which is well defined since it does not depend on the equivalence class. Let $ r \geq r_\mu + r_\nu$. Define the modifications $\tilde{\mu}, \tilde{\nu} \in M^+(\widetilde{X})$ by
    \begin{align*}
        \tilde{\mu} &= \mu + (r- r_\mu) \delta_{*}, \quad \quad \tilde{\nu} = \nu + (r-r_\nu) \delta_{*},
    \end{align*}
    where $\delta_{*}$ is the Dirac measure at the basepoint $*$. Then, $W^\partial_p[d](\mu, \nu) = W_p[d](\tilde{\mu}, \tilde{\nu})$.
\end{proposition}

\begin{corollary}
    Let $(X,d,A)$ be a metric pair and $(\widetilde{X}, \tilde{d}, *)$ be its quotient metric space. Let $q: \cM^+_{\fin, p}(X, A) \rightarrow \cM^+_{\fin,p}(\widetilde{X})$ be the bijective map induced by the quotient map. Suppose $\mu, \nu \in \cM^+_{\fin, p}(X, A)$, then $W^\partial_p[d](\mu, \nu) = W^\partial_p[\tilde{d}](q(\mu), q(\nu))$.
\end{corollary}
\begin{proof}
    Let $\tilde{\mu}, \tilde{\nu} \in M^+(\widetilde{X})$ be the modifications of $\mu$ and $\nu$ from~\Cref{prop:partial_ordinary_wass}. Then, $W^\partial_p[d](\mu, \nu) = W_p[d](\tilde{\mu}, \tilde{\nu}) = W^\partial_p[\tilde{d}](q(\mu), q(\nu))$, since $\tilde{\mu}, \tilde{\nu}$ are  modifications of $q(\mu)$ and $q(\nu)$.
\end{proof}

\section{Comparison of KME Time Complexity}\label{apx:kme_time}
In this appendix, we briefly consider the time complexity of the kernel mean embedding method considered in~\Cref{ssec:kme}. The Gram matrix computation for the KME kernel (with an underlying sliced Wasserstein kernel) scales poorly. In particular, the computation of this Gram matrix requires $O(N_{simulation}^2 T^2)$ evaluations of $\kappa_{SW}$, and each evaluation of $\kappa_{SW}$ has $O(n\log(n))$ complexity, where $n$ is the maximum cardinality of the two persistence diagrams~\cite{carriere_sliced_2017}. On the other hand, computing moments require only $O(N_{simulation} Tn)$ basic operations, and then the signature kernel requires $O(N_{simulation}^2 T^2)$ basic operations~\cite{kiraly_kernels_2019}. The result is significantly longer computation time of the KME kernel as shown in the timing experiments shown on the following tables. The following computations were done with a single thread on a 2021 Macbook Pro with M1 Pro processor.

\begin{table}[h]
    \centering
    \begin{tabular}{ccccc} 
        & $T=5$ & $T=10$ & $T=15$ & $T=20$ \\\toprule
        $N=10$ & 0.13 & 0.25 & 0.53 & 0.94 \\
        $N=20$ & 0.24 & 0.93 & 2.05 & 3.65\\
        $N=30$ & 0.53 & 2.04 & 4.77 & 8.05\\
        $N=40$ & 0.96 & 3.65 & 7.97 & 14.08\\
    \end{tabular}
\caption{Time in seconds for computing the sliced wasserstein kernel (5 slices) with $N$ simulations and $T$ time points each.}

\end{table}

\begin{table}[h]
    \centering
    \begin{tabular}{ccccc} 
        & $T=5$ & $T=10$ & $T=15$ & $T=20$ \\ \toprule
        $N=10$ & 0.0025 & 0.0021 & 0.0030 & 0.0072\\
        $N=20$ & 0.0055 & 0.0069 & 0.0140 & 0.0150\\
        $N=30$ & 0.0110 & 0.0130 & 0.0159 & 0.0350\\
        $N=40$ & 0.0266 & 0.0211 & 0.0244 & 0.0476\\
    \end{tabular}
\caption{Time in seconds for computing the signature kernel (truncation level 3) with persistence moments (truncation level 2) with $N$ simulations and $T$ time points each.}

\end{table}

\section*{Acknowledgments}
CG was funded by NSF-1854683 and AFOSR FA9550-21-1-0266. 
DL was funded by the United States Office of the Assistant Secretary of Defence Research and Engineering, ONR N00014-16-1-2010, NSERC PGS-D3 scholarship, NCCR-Synapsy Phase-3 SNSF grant number 51NF40-185897, and the Hong Kong Innovation and Technology Commission (InnoHK Project CIMDA).

\bibliographystyle{plain}
\bibliography{pdpaths}

\begin{thebibliography}{10}

\bibitem{adams_persistence_2017}
Henry Adams, Tegan Emerson, Michael Kirby, Rachel Neville, Chris Peterson,
  Patrick Shipman, Sofya Chepushtanova, Eric Hanson, Francis Motta, and Lori
  Ziegelmeier.
\newblock Persistence images: A stable vector representation of persistent
  homology.
\newblock {\em J. Mach. Learn. Res.}, 18(8):1--35, 2017.

\bibitem{adcock_ring_2016}
Aaron Adcock, Erik Carlsson, and Gunnar Carlsson.
\newblock The ring of algebraic functions on persistence bar codes.
\newblock {\em Homology, Homotopy and Applications}, 18(1):381--402, May 2016.

\bibitem{ali_survey_2022}
Dashti Ali, Aras Asaad, Maria-Jose Jimenez, Vidit Nanda, Eduardo
  {Paluzo-Hidalgo}, and Manuel {Soriano-Trigueros}.
\newblock A {{Survey}} of {{Vectorization Methods}} in {{Topological Data
  Analysis}}, December 2022.

\bibitem{aliaga_integral_2020}
Ram{\'o}n~J. Aliaga and Eva Perneck{\'a}.
\newblock Integral representation and supports of functionals on {{Lipschitz}}
  spaces.
\newblock {\em arXiv:2009.07663 [math]}, November 2020.

\bibitem{asaad_persistent_2022}
Aras Asaad, Dashti Ali, Taban Majeed, and Rasber Rashid.
\newblock Persistent {{Homology}} for {{Breast Tumor Classification Using
  Mammogram Scans}}.
\newblock {\em Mathematics}, 10(21):4039, January 2022.

\bibitem{backhoff-veraguas_all_2020}
Julio {Backhoff-Veraguas}, Daniel Bartl, Mathias Beiglb{\"o}ck, and Manu Eder.
\newblock All adapted topologies are equal.
\newblock {\em Probab. Theory Related Fields}, 178(3):1125--1172, December
  2020.

\bibitem{bhaskar_analyzing_2019}
Dhananjay Bhaskar, Angelika Manhart, Jesse Milzman, John~T. Nardini,
  Kathleen~M. Storey, Chad~M. Topaz, and Lori Ziegelmeier.
\newblock Analyzing collective motion with machine learning and topology.
\newblock {\em Chaos}, 29(12):123125, 2019.

\bibitem{boedihardjo_note_2015}
Horatio Boedihardjo, Xi~Geng, Terry Lyons, and Danyu Yang.
\newblock Note on the signatures of rough paths in a banach space.
\newblock {\em arXiv:1510.04172 [math]}, October 2015.

\bibitem{boedihardjo_signature_2016}
Horatio Boedihardjo, Xi~Geng, Terry Lyons, and Danyu Yang.
\newblock The signature of a rough path: Uniqueness.
\newblock {\em Adv. Math.}, 293:720--737, April 2016.

\bibitem{bonnier_adapted_2020}
Patric Bonnier, Chong Liu, and Harald Oberhauser.
\newblock Adapted topologies and higher rank signatures.
\newblock {\em arXiv:2005.08897 [math]}, 2020.

\bibitem{bubenik_statistical_2015-1}
Peter Bubenik.
\newblock Statistical topological data analysis using persistence landscapes.
\newblock {\em J. Mach. Learn. Res.}, 16:77--102, 2015.

\bibitem{bubenik_universality_2020}
Peter Bubenik and Alex Elchesen.
\newblock Universality of persistence diagrams and the bottleneck and
  {{Wasserstein}} distances.
\newblock {\em arXiv:1912.02563 [cs, math]}, September 2020.

\bibitem{bubenik_virtual_2020}
Peter Bubenik and Alex Elchesen.
\newblock Virtual persistence diagrams, signed measures, {{Wasserstein}}
  distances, and {{Banach}} spaces.
\newblock {\em arXiv:2012.10514 [math]}, December 2020.

\bibitem{bubenik_embeddings_2019}
Peter Bubenik and Alexander Wagner.
\newblock Embeddings of persistence diagrams into {{Hilbert}} spaces.
\newblock {\em arXiv:1905.05604 [cs, math, stat]}, May 2019.

\bibitem{carriere_metric_2019}
Mathieu Carri{\`e}re and Ulrich Bauer.
\newblock On the {{Metric Distortion}} of {{Embedding Persistence Diagrams}}
  into {{Separable Hilbert Spaces}}.
\newblock In Gill Barequet and Yusu Wang, editors, {\em 35th {{International
  Symposium}} on {{Computational Geometry}} ({{SoCG}} 2019)}, volume 129 of
  {\em Leibniz {{International Proceedings}} in {{Informatics}} ({{LIPIcs}})},
  pages 21:1--21:15, {Dagstuhl, Germany}, 2019. {Schloss Dagstuhl\textendash
  Leibniz-Zentrum fuer Informatik}.

\bibitem{carriere_sliced_2017}
Mathieu Carri{\`e}re, Marco Cuturi, and Steve Oudot.
\newblock Sliced {{Wasserstein}} kernel for persistence diagrams.
\newblock In {\em Proceedings of the 34th {{International Conference}} on
  {{Machine Learning}}}, pages 664--673, July 2017.

\bibitem{cass_computing_2021}
Thomas Cass, James Foster, Terry Lyons, Cristopher Salvi, and Weixin Yang.
\newblock Computing the untruncated signature kernel as the solution of a
  {{Goursat}} problem.
\newblock {\em arXiv:2006.14794 [cs, math]}, January 2021.

\bibitem{chazal_observable_2016}
Fr{\'e}d{\'e}ric Chazal, William {Crawley-Boevey}, and Vin {de Silva}.
\newblock The observable structure of persistence modules.
\newblock {\em Homology, Homotopy Appl.}, 18(2):247--265, December 2016.

\bibitem{chazal_structure_2016}
Fr{\'e}d{\'e}ric Chazal, Vin de~Silva, Marc Glisse, and Steve Oudot.
\newblock {\em The {{Structure}} and {{Stability}} of {{Persistence Modules}}}.
\newblock Springer {{Briefs}} in {{Mathematics}}. {Springer International
  Publishing}, 2016.

\bibitem{chen_integration_1958}
Kuo-Tsai Chen.
\newblock Integration of paths -- a faithful representation of paths by
  noncommutative formal power series.
\newblock {\em Trans. Amer. Math. Soc.}, 89(2):395--407, 1958.

\bibitem{chen_iterated_1977-1}
Kuo-Tsai Chen.
\newblock Iterated path integrals.
\newblock {\em Bull. Amer. Math. Soc.}, 83(5):831--879, 1977.

\bibitem{chevyrev_primer_2016}
Ilya Chevyrev and Andrey Kormilitzin.
\newblock A primer on the signature method in machine learning.
\newblock {\em arXiv:1603.03788 [cs, stat]}, 2016.

\bibitem{chevyrev_persistence_2020}
Ilya Chevyrev, Vidit Nanda, and Harald Oberhauser.
\newblock Persistence paths and signature features in topological data
  analysis.
\newblock {\em IEEE Trans. Pattern Anal. Mach. Intell.}, 42(1):192--202, 2020.

\bibitem{chevyrev_signature_2022}
Ilya Chevyrev and Harald Oberhauser.
\newblock Signature {{Moments}} to {{Characterize Laws}} of {{Stochastic
  Processes}}.
\newblock {\em Journal of Machine Learning Research}, 23(176):1--42, 2022.

\bibitem{chistyakov_maps_1998}
V.~V. Chistyakov and O.~E. Galkin.
\newblock On maps of bounded p-variation with p{$>$}1.
\newblock {\em Positivity}, 2(1):19--45, March 1998.

\bibitem{chung_persistence_2022}
Yu-Min Chung and Austin Lawson.
\newblock Persistence {{Curves}}: {{A}} canonical framework for summarizing
  persistence diagrams.
\newblock {\em Advances in Computational Mathematics}, 48(1):6, January 2022.

\bibitem{cohen-steiner_vines_2006}
David {Cohen-Steiner}, Herbert Edelsbrunner, and Dmitriy Morozov.
\newblock Vines and vineyards by updating persistence in linear time.
\newblock In {\em Proceedings of the 22nd {{Annual Symposium}} on
  {{Computational Geometry}}}, {{SCG}} '06, pages 119--126, {New York, NY,
  USA}, 2006. {ACM}.

\bibitem{cuth_structure_2016}
Marek C{\'u}th, Michal Doucha, and Przemys{\l}aw Wojtaszczyk.
\newblock On the structure of {{Lipschitz-free}} spaces.
\newblock {\em Proc. Amer. Math. Soc.}, 144(9):3833--3846, 2016.

\bibitem{cuth_isometric_2017}
Marek C{\'u}th, Ond{\v r}ej F.~K. Kalenda, and Petr Kaplick{\'y}.
\newblock Isometric representation of {{Lipschitz-free}} spaces over convex
  domains in finite dimensional spaces.
\newblock {\em Mathematika}, 63(2):538--552, 2017/ed.

\bibitem{dey_computational_2022-1}
Tamal~Krishna Dey and Yusu Wang.
\newblock {\em Computational {{Topology}} for {{Data Analysis}}}.
\newblock {Cambridge University Press}, {Cambridge}, 2022.

\bibitem{divol_understanding_2021}
Vincent Divol and Th{\'e}o Lacombe.
\newblock Understanding the topology and the geometry of the space of
  persistence diagrams via optimal partial transport.
\newblock {\em J. Appl. Comput. Topol.}, 5(1):1--53, March 2021.

\bibitem{divol_choice_2019}
Vincent Divol and Wolfgang Polonik.
\newblock On the choice of weight functions for linear representations of
  persistence diagrams.
\newblock {\em J. Appl. Comput. Topol.}, 3(3):249--283, September 2019.

\bibitem{dorsogna_self-propelled_2006}
M.~R. D'Orsogna, Y.~L. Chuang, A.~L. Bertozzi, and L.~S. Chayes.
\newblock Self-propelled particles with soft-core interactions: Patterns,
  stability, and collapse.
\newblock {\em Phys. Rev. Lett.}, 96(10), March 2006.

\bibitem{drucker_support_1997}
Harris Drucker, Christopher J.~C. Burges, Linda Kaufman, Alex Smola, and
  Vladimir Vapnik.
\newblock Support vector regression machines.
\newblock In M.~C. Mozer, M.~Jordan, and T.~Petsche, editors, {\em Advances in
  {{Neural Information Processing Systems}}}, volume~9, pages 155--161. {MIT
  Press}, 1997.

\bibitem{edelsbrunner_persistent_2008}
H.~Edelsbrunner and J.~Harer.
\newblock Persistent homology - a survey.
\newblock In {\em Surveys on Discrete and Computational Geometry}, volume 453,
  pages 257--282, 2008.

\bibitem{fermanian_embedding_2021}
Adeline Fermanian.
\newblock Embedding and learning with signatures.
\newblock {\em Comput. Statist. Data Anal.}, 157:107148, May 2021.

\bibitem{figalli_new_2010}
Alessio Figalli and Nicola Gigli.
\newblock A new transportation distance between non-negative measures, with
  applications to gradients flows with {{Dirichlet}} boundary conditions.
\newblock {\em J. Math. Pures Appl.}, 94(2):107--130, 2010.

\bibitem{friz_course_2020}
Peter~K. Friz and Martin Hairer.
\newblock {\em A {{Course}} on {{Rough Paths}}: {{With}} an {{Introduction}} to
  {{Regularity Structures}}}.
\newblock Universitext. {Springer International Publishing}, second edition,
  2020.

\bibitem{friz_multidimensional_2010}
Peter~K. Friz and Nicolas~B. Victoir.
\newblock {\em Multidimensional {{Stochastic Processes}} as {{Rough Paths}}:
  {{Theory}} and {{Applications}}}.
\newblock Cambridge {{Studies}} in {{Advanced Mathematics}}. {Cambridge
  University Press}, 2010.

\bibitem{fukumizu_learning_2011}
Kenji Fukumizu, Gert Lanckriet, and Bharath~K. Sriperumbudur.
\newblock Learning in {{Hilbert}} vs. {{Banach}} spaces: {{A}} measure
  embedding viewpoint.
\newblock In J.~{Shawe-Taylor}, R.~Zemel, P.~Bartlett, F.~Pereira, and K.~Q.
  Weinberger, editors, {\em Advances in {{Neural Information Processing
  Systems}}}, volume~24, pages 1773--1781. {Curran Associates, Inc.}, 2011.

\bibitem{ghrist_barcodes:_2008-3}
Robert Ghrist.
\newblock Barcodes: {{The}} persistent topology of data.
\newblock {\em Bull. Amer. Math. Soc.}, 45(1):61--75, 2008.

\bibitem{giles_generalization_1971}
Robin Giles.
\newblock A generalization of the strict topology.
\newblock {\em Trans. Amer. Math. Soc.}, 161:467--474, 1971.

\bibitem{giusti_iterated_2020}
Chad Giusti and Darrick Lee.
\newblock Iterated integrals and population time series analysis.
\newblock In {\em Topological {{Data Analysis}}}, Abel {{Symposia}}, pages
  219--246. {Springer International Publishing}, {Cham}, 2020.

\bibitem{hambly_uniqueness_2010}
Ben Hambly and Terry Lyons.
\newblock Uniqueness for the signature of a path of bounded variation and the
  reduced path group.
\newblock {\em Ann. of Math.}, 171(1):109--167, 2010.

\bibitem{henselman_matroid_2017}
Gregory Henselman and Robert Ghrist.
\newblock Matroid filtrations and computational persistent homology.
\newblock {\em arXiv:1606.00199 [math]}, October 2017.

\bibitem{kalisnik_tropical_2019}
Sara Kali{\v s}nik.
\newblock Tropical coordinates on the space of persistence barcodes.
\newblock {\em Found. Comput. Math.}, 19(1):101--129, February 2019.

\bibitem{khasawneh_chatter_2016}
Firas~A. Khasawneh and Elizabeth Munch.
\newblock Chatter detection in turning using persistent homology.
\newblock {\em Mech. Syst. Signal Process.}, 70--71:527--541, March 2016.

\bibitem{kidger_deep_2019}
Patrick Kidger, Patric Bonnier, Imanol Perez~Arribas, Cristopher Salvi, and
  Terry Lyons.
\newblock Deep {{Signature Transforms}}.
\newblock In {\em Advances in {{Neural Information Processing Systems}}},
  volume~32, pages 3099--3109. {Curran Associates, Inc.}, 2019.

\bibitem{kim_spatiotemporal_2020}
Woojin Kim and Facundo M{\'e}moli.
\newblock Spatiotemporal persistent homology for dynamic metric spaces.
\newblock {\em Discrete Comput. Geom.}, February 2020.

\bibitem{kim_generalized_2021}
Woojin Kim and Facundo M{\'e}moli.
\newblock Generalized persistence diagrams for persistence modules over posets.
\newblock {\em Journal of Applied and Computational Topology}, 5(4):533--581,
  December 2021.

\bibitem{kiraly_kernels_2019}
Franz~J. Kiraly and Harald Oberhauser.
\newblock Kernels for sequentially ordered data.
\newblock {\em J. Mach. Learn. Res.}, 20(31):1--45, 2019.

\bibitem{kwitt_statistical_2015}
Roland Kwitt, Stefan Huber, Marc Niethammer, Weili Lin, and Ulrich Bauer.
\newblock Statistical topological data analysis - a kernel perspective.
\newblock In C.~Cortes, N.~D. Lawrence, D.~D. Lee, M.~Sugiyama, and R.~Garnett,
  editors, {\em Advances in {{Neural Information Processing Systems}}}, pages
  3070--3078. {Curran Associates, Inc.}, 2015.

\bibitem{lee_path_2020}
Darrick Lee and Robert Ghrist.
\newblock Path {{Signatures}} on {{Lie Groups}}.
\newblock {\em arXiv:2007.06633 [cs, math, stat]}, 2020.

\bibitem{lesnick_theory_2015}
Michael Lesnick.
\newblock The {{Theory}} of the {{Interleaving Distance}} on {{Multidimensional
  Persistence Modules}}.
\newblock {\em Foundations of Computational Mathematics}, 15(3):613--650, June
  2015.

\bibitem{lyons_differential_1998}
Terry Lyons.
\newblock Differential equations driven by rough signals.
\newblock {\em Rev. Mat. Iberoam.}, 14(2):215--310, 1998.

\bibitem{lyons_differential_2007}
Terry Lyons, Michael Caruana, and Thierry L{\'e}vy.
\newblock {\em Differential {{Equations Driven}} by {{Rough Paths}}}.
\newblock \'Ec. {{\'Et\'e Probab}}. {{St}}.-{{Flour}}. {Springer-Verlag},
  {Berlin Heidelberg}, 2007.

\bibitem{lyons_system_2007}
Terry Lyons and Zhongmin Qian.
\newblock {\em System {{Control}} and {{Rough Paths}}}.
\newblock {Clarendon}, {Oxford}, 2007.

\bibitem{mileyko_probability_2011}
Yuriy Mileyko, Sayan Mukherjee, and John Harer.
\newblock Probability measures on the space of persistence diagrams.
\newblock {\em Inverse Problems}, 27(12):124007, 2011.

\bibitem{morrill_generalised_2021}
James Morrill, Adeline Fermanian, Patrick Kidger, and Terry Lyons.
\newblock A {{Generalised Signature Method}} for {{Multivariate Time Series
  Feature Extraction}}.
\newblock {\em arXiv:2006.00873 [cs, stat]}, February 2021.

\bibitem{muandet_learning_2012}
Krikamol Muandet, Kenji Fukumizu, Francesco Dinuzzo, and Bernhard
  Sch{\"o}lkopf.
\newblock Learning from distributions via support measure machines.
\newblock In F.~Pereira, C.J. Burges, L.~Bottou, and K.Q. Weinberger, editors,
  {\em Advances in Neural Information Processing Systems}, volume~25. {Curran
  Associates, Inc.}, 2012.

\bibitem{munch_applications_2013}
Elizabeth Munch.
\newblock {\em Applications of {{Persistent Homology}} to {{Time Varying
  Systems}}}.
\newblock PhD thesis, Duke University, 2013.

\bibitem{munch_probabilistic_2015}
Elizabeth Munch, Katharine Turner, Paul Bendich, Sayan Mukherjee, Jonathan
  Mattingly, and John Harer.
\newblock Probabilistic {{Fr\'echet}} means for time varying persistence
  diagrams.
\newblock {\em Electron. J. Stat.}, 9(1):1173--1204, 2015.

\bibitem{nguyen_thermal_2012}
Nguyen H.~P. Nguyen, Eric Jankowski, and Sharon~C. Glotzer.
\newblock Thermal and athermal three-dimensional swarms of self-propelled
  particles.
\newblock {\em Phys. Rev. E}, 86(1):011136, July 2012.

\bibitem{ostrovska_generalized_2019}
Sofiya Ostrovska and Mikhail Ostrovskii.
\newblock Generalized transportation cost spaces.
\newblock {\em Mediterr. J. Math.}, 16(6):157, December 2019.

\bibitem{otter_roadmap_2017}
Nina Otter, Mason~A. Porter, Ulrike Tillmann, Peter Grindrod, and Heather~A.
  Harrington.
\newblock A roadmap for the computation of persistent homology.
\newblock {\em EPJ Data Sci.}, 6(1):1--38, December 2017.

\bibitem{perea_sw1pers_2015}
Jose~A. Perea, Anastasia Deckard, Steve~B. Haase, and John Harer.
\newblock {{SW1PerS}}: {{Sliding}} windows and 1-persistence scoring;
  discovering periodicity in gene expression time series data.
\newblock {\em BMC Bioinform,}, 16(1):257, August 2015.

\bibitem{perea_sliding_2015}
Jose~A. Perea and John Harer.
\newblock Sliding windows and persistence: {{An}} application of topological
  methods to signal analysis.
\newblock {\em Found. Comput. Math.}, 15(3):799--838, June 2015.

\bibitem{pun_persistent-homology-based_2018}
Chi~Seng Pun, Kelin Xia, and Si~Xian Lee.
\newblock Persistent-{{Homology-Based Machine Learning}} and {{Its
  Applications}} -- {{A Survey}}, October 2018.

\bibitem{reininghaus_stable_2015}
J.~Reininghaus, S.~Huber, U.~Bauer, and R.~Kwitt.
\newblock A stable multi-scale kernel for topological machine learning.
\newblock In {\em 2015 {{IEEE Conference}} on {{Computer Vision}} and {{Pattern
  Recognition}}}, pages 4741--4748, June 2015.

\bibitem{rieck_uncovering_2020}
Bastian Rieck, Tristan Yates, Christian Bock, Karsten Borgwardt, Guy Wolf,
  Nicholas {Turk-Browne}, and Smita Krishnaswamy.
\newblock Uncovering the topology of time-varying {{fMRI}} data using cubical
  persistence.
\newblock In H.~Larochelle, M.~Ranzato, R.~Hadsell, M.~F. Balcan, and H.~Lin,
  editors, {\em Advances in Neural Information Processing Systems}, volume~33,
  pages 6900--6912. {Curran Associates, Inc.}, 2020.

\bibitem{schlegel_approximate_2020}
Kevin Schlegel.
\newblock Approximate representer theorems in non-reflexive {{Banach}} spaces.
\newblock In {\em Proceedings of the 31st {{International Conference}} on
  {{Algorithmic Learning Theory}}}, pages 827--844. {PMLR}, January 2020.

\bibitem{scholkopf_learning_2002}
B.~Sch{\"o}lkopf and A.J. Smola.
\newblock {\em Learning with Kernels: {{Support}} Vector Machines,
  Regularization, Optimization, and Beyond}.
\newblock Adaptive Computation and Machine Learning. {MIT Press}, 2002.

\bibitem{simon-gabriel_kernel_2018}
Carl-Johann {Simon-Gabriel} and Bernhard Sch{\"o}lkopf.
\newblock Kernel distribution embeddings: Universal kernels, characteristic
  kernels and kernel metrics on distributions.
\newblock {\em J. Mach. Learn. Res.}, 19(44):1--29, 2018.

\bibitem{tauzin_giotto-tda_2021}
Guillaume Tauzin, Umberto Lupo, Lewis Tunstall, Julian~Burella P{\'e}rez,
  Matteo Caorsi, Anibal~M. {Medina-Mardones}, Alberto Dassatti, and Kathryn
  Hess.
\newblock Giotto-tda: : {{A Topological Data Analysis Toolkit}} for {{Machine
  Learning}} and {{Data Exploration}}.
\newblock {\em Journal of Machine Learning Research}, 22(39):1--6, 2021.

\bibitem{topaz_topological_2015}
Chad~M. Topaz, Lori Ziegelmeier, and Tom Halverson.
\newblock Topological data analysis of biological aggregation models.
\newblock {\em PLoS One}, 10(5):e0126383, 2015.

\bibitem{tralie_quasiperiodicity_2018}
Christopher~J. Tralie and Jose~A. Perea.
\newblock ({{Quasi}})periodicity quantification in video data, using topology.
\newblock {\em SIAM J. Imaging Sci.}, 11(2):1049--1077, January 2018.

\bibitem{turner_same_2020}
Katharine Turner and Gard Spreemann.
\newblock Same but different: Distance correlations between topological
  summaries.
\newblock In Nils~A. Baas, Gunnar~E. Carlsson, Gereon Quick, Markus Szymik, and
  Marius Thaule, editors, {\em Topological {{Data Analysis}}}, Abel
  {{Symposia}}, pages 459--490. {Springer International Publishing}, {Cham},
  2020.

\bibitem{wagner_nonembeddability_2021}
Alexander Wagner.
\newblock Nonembeddability of persistence diagrams with p {$>$}2
  {{Wasserstein}} metric.
\newblock {\em Proceedings of the American Mathematical Society}, 2021.

\bibitem{weaver_lipschitz_2018}
Nik Weaver.
\newblock {\em Lipschitz Algebras}.
\newblock {World Scientific}, second edition, 2018.

\bibitem{xian_capturing_2020}
Lu~Xian, Henry Adams, Chad~M. Topaz, and Lori Ziegelmeier.
\newblock Capturing dynamics of time-varying data via topology.
\newblock {\em arXiv:2010.05780 [cs, math, stat]}, October 2020.

\bibitem{yoo_topological_2016}
Jaejun Yoo, Eun~Young Kim, Yong~Min Ahn, and Jong~Chul Ye.
\newblock Topological persistence vineyard for dynamic functional brain
  connectivity during resting and gaming stages.
\newblock {\em J. Neurosci. Methods}, 267:1--13, July 2016.

\bibitem{zhang_reproducing_2009}
Haizhang Zhang, Yuesheng Xu, and Jun Zhang.
\newblock Reproducing kernel {{Banach}} spaces for machine learning.
\newblock {\em J. Mach. Learn. Res.}, 10(95):2741--2775, 2009.

\bibitem{zomorodian_computing_2005}
Afra Zomorodian and Gunnar Carlsson.
\newblock Computing {{Persistent Homology}}.
\newblock {\em Discrete Comput. Geom.}, 33(2):249--274, February 2005.

\end{thebibliography}

\end{document}